%% file: Classification_v6-5.tex
\documentclass[11pt]{amsart}
\usepackage[toc,page]{appendix}
\usepackage[utf8x]{inputenc}

\usepackage{amsmath, amsthm, amsfonts, amssymb}
\usepackage{epsfig, enumerate}
\usepackage{color}
\usepackage{stmaryrd}

\usepackage{hyperref}

\input{makros.tex}
\def \Diff{{\rm Diff}}
\def \tal{{\tilde{\alpha}}}
\def \tM{{\tilde{M}}}
\def \al{{\alpha}}

\def \W{{\mathcal W}}
\def \O{{\mathcal{O}}}
\def \A{{\mathcal A}}

\def \dist{{\mathrm{dist}}}
\def \C{{\mathcal{C}}}
\def \cF{{\mathcal{F}}}
\def \cE{{\mathcal{E}}}
\def \cR{{\mathcal{R}}}
\def \cV{{\mathcal{V}}}
\def \tE{{\tilde E}}
\def \tcE{{\tilde {\mathcal E}}}
\def \mC{{\mathfrak C}}
\def \H{{\mathcal{H}}}

\def \bal{\bar\alpha}
\def \cU{\mathcal{U}}

\numberwithin{theorem}{section}
\numberwithin{equation}{section}

\author{Danijela Damjanovi\'c}

\address[Damjanovi\'c]{Department of mathematics, Kungliga Tekniska högskolan, Lindstedtsvägen 25, SE-100 44 Stockholm, Sweden.} 
\email{ddam@kth.se}
\thanks{The first aurthor is supported by the Swedish Research Council grant  2015-04644}

\author{Disheng Xu}

\address[Xu]{Department of mathematics, the University of Chicago, Chicago, IL, US, 60637}
\email{dishengxu@math.uchicago.edu}


\begin{document}

\title[Classification of Abelian action]{On classification of higher rank Anosov actions on compact manifold}

\date{\today}
\maketitle
{\centering\footnotesize \emph{Dedicated to Anatole Katok}\par}
\begin{abstract} We prove global smooth classification results for Anosov $\mathbb Z^k$ actions on \emph{general} compact manifolds, under certain irreduciblity conditions and the presence of sufficiently many Anosov elements. In particular we remove all the \emph{uniform control} assumptions which were used in all the previous results towards Katok-Spatzier global rigidity conjecture on general manifolds. 
The  main idea is to create a new mechanism labelled \emph{non-uniform redefining argument}, to prove  continuity of certain dynamically-defined objects. This leads to uniform control for a higher rank action and should apply to more general rigidity problems in dynamical systems.
\end{abstract}
 
\tableofcontents
\section{Introduction}
Anosov diffeomorphisms are well-studied class of systems, which in many respects either have or are expected to have rigid dynamical features.  For example, it is well known result of Anosov  that complete \emph{topological} orbit structure of an Anosov diffeomorphism is preserved under $C^1$-small perturbations.  
Principal examples of Anosov diffeomorphisms are hyperbolic affine maps on nilmanifolds (in particular on tori), and manifolds finitely covered by nilmanifolds (infranilmanifolds). 
A long outstanding global \emph{topological} rigidity conjecture says that all Anosov diffeomorphisms are \emph{topologically} conjugate to affine maps on infranilmanifolds. 
This conjecture on nilmanifolds has been proved by Franks and Manning \cite{F}, \cite{M}, but only a few results are known for general manifolds, (cf. \cite{GH}, \cite{BM} and the references therein).

These topological local and global rigidity results for Anosov diffeomorphisms cannot be improved in general to \emph{smooth} rigidity results. It is easy to see already for algebraic Anosov maps on the torus that the \emph{differentiable} orbit  structure may be destroyed via small perturbations \cite[Section 2.1]{KH}\color{black}


 
It was a remarkable discovery made by Katok and Spatzier \cite{KatSpa}, that when certain \emph{commuting} algebraic Anosov maps generate a  $\ZZ^k$ group action, $k\ge 2$, then perturbations of this \emph{group action} indeed preserve full \emph{differentiable} (in fact smooth) orbit structure. This phenomenon is labelled \emph{local rigidity}. Actions which have such strong rigidity property of course must not reduce in any way to $\mathbb Z$-actions, i.e. single diffeomorphisms. For example if we pick two Anosov diffeomorphisms on two manifolds, this  naturally induces a $\ZZ^2$ Anosov action on the product manifold. However, this "product action" cannot enjoy any rigidity whenever single Anosov diffeomorphism does not. 
Therefore, some "irreducibility" condition on a higher-rank action is necessary. \color{black} 
All known ``irreducible" higher rank Anosov actions are $C^\infty-$conjugate to algebraic models, and "irreducibility" for algebraic actions is a well understood notion. 

All the strong rigidity properties (found in \cite{KatSpa, KalKat, KL, KL96}, etc.) of the irreducible algebraic models support the following global smooth classification conjecture for \emph{higher-rank Anosov actions}, i.e. $\ZZ^k$ actions ($k\geq 2$) containing one Anosov diffeomorphism, made by A. Katok and R. Spatzier:\\
\\
\textbf{Katok-Spatzier global rigidity conjecture:}
\emph{All ``irreducible" smooth Anosov $\ZZ^k$ actions for $k\ge 2$, on any compact smooth manifold, are $C^\infty$-conjugate to algebraic models on infranilmanifolds. }

Significant progress (\cite{FKS, FKS1, HW, KalSpa, KS06, KS07, H07, DX17}, etc.) has been made towards proving the Katok-Spatzier conjecture. In particular,  Rodriguez Hertz and Wang \cite{HW} proved the Katok-Spatzier conjecture under the assumption that the action is on a nilmanifold.  This result was preceded by the work of Fisher, Kalinin and Spatzier \cite{FKS} on nilmanifolds under the assumption that the action has sufficiently many Anosov elements.
The crucial starting point for these works on nilmanifolds actions is the Franks-Manning topological conjugacy (linearization) from the Anosov action to the algebraic model (\cite{F, M}), and in this context ``irreducibility" condition can be easily defined as the lack of rank-one factors for the linearization.

For higher rank Anosov actions on \textbf{general} smooth manifolds, the Katok-Spatzier conjecture is wide open and only a few partial results are known, (cf. \cite{KalSpa, KS06, KS07}). It is even not known how to formulate it explicitly. The reason is that it is hard to give a precise definition of ``irreducible", because on general manifolds \emph{a priori} there is no topological conjugacy from the action to an algebraic model. 

In all the existing results for Katok-Spatzier conjecture on general manifolds, the following three types of conditions are always assumed:
\begin{itemize}
\item \emph{Uniform control} for the action on certain dynamically-defined foliations. (For example uniform quasi-conformality , cf. \cite{KS06, KS07, DX17}, etc)\\
\item Certain irreducibility conditions.  
For example, \emph{totally non-symplectic} condition (TNS), cf. section \ref{sec: sett stat}. \\
\item The presence of sufficiently many Anosov elements, (for example \emph{totally Anosov} condition, cf. section \ref{sec: sett stat}), is assumed in almost all the existing results on the Katok-Spatzier conjecture (cf. \cite{FKS, FKS1, KalSpa, KS06, KS07, DX17}, etc). 
\\

\end{itemize}
The TNS condition was defined by Katok, Nitica and Torok \cite{KNT}  for the purpose of obtaining a more geometric approach to proving certain rigidity properties of higher-rank Anosov actions. It was initially used for algebraic actions, but it is of great importance in understanding ``irreducibility" for general non-algebraic actions on general manifolds. 

In this paper, we prove Katok-Spatzier conjecture on general manifolds under certain \color{black} irreducibility condition and the assumption of presence of many Anosov elements. In particular one consequence of our results is that\\
\\
\textbf{``Theorem 0".} \emph{ All the uniform control assumptions for the Anosov actions can be \textbf{removed} from all the previous results towards  the Katok-Spatzier conjecture on general manifolds.}

For precise statements of the main results on global rigidity which imply Theorem 0, cf. section \ref{sec: main thm rig chp 2}. 

A surprising consequence of our results is also the following \emph{measure rigidity} result:
\begin{maintheorem}\label{coro: conserv intro}Any TNS totally Anosov action $\al$ preserves a smooth volume form.
\end{maintheorem}
Theorem \ref{coro: conserv intro} comes as a surprise given that it is clearly not true for a single Anosov diffeomorphism. In recent breakthrough concerning Zimmer's conjecture by \cite{BFH16, BFH17} given various strong rigidity
properties of higher rank simple groups, a measure rigidity result (see also \cite{BRW}) is proved and it plays a crucial role.

The group $\ZZ^k$ on the other hand does not enjoy as many rigidity properties as higher rank simple groups and we use here a completely different approach than that in previous measure rigidity arguments, in particular we use no entropy arguments.

We trust that our results and techniques will have further applications towards Katok-Spatzier conjecture and more general smooth rigidity problems in the future. The basic philosophy hidden behind our approach can be summarised as follows (similar approach has already been used for certain higher rank semisimple Lie group actions, cf. \cite{BFH16, BFH17}):
\begin{itemize}
\item \emph{Any measurable dynamically-defined object for a higher-rank group action with some hyperbolicity, should be H\"older continuous.}
\item \emph{Any non-uniform estimate for dynamically-defined quantities of  a higher-rank group action with some hyperbolicity, should be uniform.}
\end{itemize}
\color{black}

As already mentioned, one of the main results we prove in this paper is towards the second item above, namely: \\
\\
\textbf{Theorem 2'} \emph{For any TNS totally Anosov action $\al$, up to a polynomial factor,  the derivative $D\al$ has exactly the same ``growth speed''} \textbf{everywhere}.\\

Again this uniformity result does not hold for single Anosov diffeomorphism! More precisely, we prove the following 
 \emph{standard form} for $D\al$. 
\begin{maintheorem}\label{main: standard form intro}For any TNS totally Anosov action $\al$, there exist $n\in \ZZ^+$, a finite cover $\bar{M}$ of $M$ and  \textbf{H\"older continuous} change of coordinates on $T\bar M$, linear on the fibers such that for any $a\in \ZZ^k$ and $x\in \bar M$, we have $$D\bal(na)(x)=
\begin{pmatrix}A_1(x)&&&\\
&A_2(x)&&\\
&&\ddots&\\
&&&A_k(x)
\end{pmatrix}$$ where $$A_i=\begin{pmatrix}
e^{n\chi_i(a)}\cdot O_{i1}(x)&\ast&\cdots&\ast\\
&e^{n\chi_i(a)}\cdot O_{i2}(x)&\cdots&\ast\\
&&\ddots&\vdots\\
&&&e^{n\chi_i(a)}\cdot O_{il}(x)
\end{pmatrix}$$
Here $\bal$ is the lift of $\al|_{n\cdot \ZZ^k}$ on $\bar M$ and each $\chi_i$ is a linear functional on $\ZZ^k$ \textbf{independent} of $x$ and $O_{ij}$ are orthogonal matrices.
\end{maintheorem}


\color{black}
The key point to proving Theorem \ref{main: standard form intro} is \textit{H\"older continuity} of the associated \textit{Zimmer measurable amenable reduction} for $D\al$. The main idea is to create a new and surprisingly effective mechanism, labelled \emph{non-uniform redefining argument}, to prove the H\"older continuity of general dynamically-defined object on the  fiber bundles over spaces with certain \emph{hyperbolicity}. 




Another key step in the proof of  Theorem \ref{main: standard form intro} is new and crucial subexponential estimate (section \ref{section: uniform estimate}) derived from our subtle cocycle-theoretical argument. Previously the corresponding estimate was only known (cf. \cite{FKS1}) for Anosov actions on nilmanifolds. We generalize it to arbitrary compact manifold.\color{black} \\
\\
\textbf{Acknowledgement:} We would like to thank David Fisher, Ralf Spatzier, Amie Wilkinson for useful discussions and Andy Hammerlindl,  Yi Pan for the help on a reference. We thank Boris Kalinin for useful discussions and for pointing out an error in the first version of the paper. Special thanks to Royal Institute of Technology in Stockholm for hospitality during the visit of the second author.

\section{Setting and statements}\label{sec: sett stat}

\subsection{Anosov $\ZZ^k$ actions on infranilmanifolds}
Suppose $M$ is a compact smooth manifold.
Recall that $f\in \Diff^1(M)$ is called \emph{Anosov} if there is a $Df$-invariant splitting $TM = E^{s} \oplus E^{u}$ of the tangent bundle of $M$ such that for some $k \geq 1$, any $x \in M$, and any choice of unit vectors $v^{s} \in E^{s}_{x}$, $v^{u} \in E^{u}_{x}$,
\[
\|Df^{k}(v^{s})\| < 1< \|Df^{k}(v^{u})\|,
\]

Now we consider a $\ZZ^k-$action $\al$ on a compact manifold $M$ by diffeomorphisms. The action is called \emph{Anosov} if there is an element that acts as an Anosov diffeomorphism. 

Recall that a compact nilmanifold is a quotient of a simply connected nilpotent Lie group $G$ by a cocompact discrete subgroup $\Gamma$,
and a compact infranilmanifold is a manifold that is finitely covered by a compact nilmanifold. A linear automorphism of a nilmanifold $G/\Gamma$ is a homeomorphism that is the projection of some $\Gamma$-preserving automorphism of $G$. An affne automorphism of $G/\Gamma$ is the composition of a linear automorphism of $G/\Gamma$ and a left translation. An affne automorphism of a compact infranilmanifold is a homeomorphism that lifts to an affne nilmanifold automorphism on a finite cover. All currently known examples of Anosov diffeomorphisms are topologically conjugate to affine automorphisms of infranilmanifolds.

\subsection{Lyapunov distributions and irreducibility conditions} \label{irr_conditions}Suppose $\mu$ is an ergodic probability measure for an Anosov $\ZZ^k$ action $\al$ on a compact manifold $M$. By commutativity, the Lyapunov decompositions for individual elements (cf. \cite{Ose}) of $\ZZ^k$ can be refined to a joint $\al-$invariant splitting. By multiplicative ergodic theorem \cite{Ose} there are finitely many linear functionals $\chi$ on $\ZZ^k$, a $\mu$ full measure set $P$, and an $\al$-invariant measurable splitting of the tangent bundle $TM =\oplus E_\chi$
over $P$ such that for all $a\in\ZZ^k$
and $v\in E_\chi$, the
Lyapunov exponent of $v$ is $\chi(a)$,
The splitting $\oplus E_\chi$ is called the \emph{(Oseledec) Lyapunov decomposition}, and the linear functionals
$\chi$ are called the \emph{Lyapunov functionals (exponents) of $\al$} (with respect to $\mu$). The hyperplanes $\ker_\chi\subset \RR^k$ are called
the \emph{Lyapunov hyperplanes}, and the connected components of $\RR^k-\cup_\chi \ker_\chi$are called the \emph{Weyl chambers} of $\al$. The elements in the union of the Weyl chambers are called \textit{regular}. 

For any Lyapunov functional $\chi$ the \textit{coarse Lyapunov distribution} is the direct sum of all Lyapunov spaces with Lyapunov functionals positively proportional to $\chi$: $E^\chi:=\oplus E_{\chi'}, \chi'=c\chi, c>0$. 



The following properties of Anosov actions have been used in a large body of work to describe irreducibility of the action, and they will appear in the main results of the paper:

\begin{definition}
Suppose $\al$ is a smooth $\ZZ^k$ Anosov action on $M$. $\al$ is called \emph{totally non-symlpectic} (TNS) if there is a fully supported $\al-$invariant ergodic measure $\mu$ such that there is no negatively proportional Lyapunov exponents of $\al$ with respect to $\mu$. Or equivalently, for any pair of Lyapunov distributions $E_\chi, E_{\chi'}$ there exists $a\in \ZZ^k$ such that $E_\chi, E_{\chi'}\subset E^s_a$.
\end{definition}

In the study of rigidity for Anosov actions it is quite natural to assume the presence of many Anosov elements, cf. \cite{FKS} and the references therein. In the rest of the paper we always consider  Anosov actions under the following assumption.

\begin{definition}\label{def: totally Ano}A $\ZZ^k$ action $\al$ on a compact manifold $M$ is called \textit{totally Anosov} if all non-trivial elements are Anosov.
\end{definition}
\color{black}

In the presence of sufficiently many Anosov elements, (e.g. totally Anosov action) and if the ergodic measure is of full support (such a measure always exists if there is a transitive Anosov element in the action) the coarse Lyapunov distributions are minimal non-trivial intersections of stable distributions for various elements of the action. They are well defined everywhere, H\"older continuous, and tangent to foliations with smooth leaves.  (For more details Section 2.2 in \cite{KalSpa}). 
Moreover, for any other action invariant measure of full support, and with Anosov elements in each Weyl chamber, the coarse Lyapunov distributions will be the same, as well as the Weyl chamber picture \cite{KalSpa}. 

A remarkable property for TNS totally Anosov action is the following lemma:
\begin{lemma}\label{lemma: elem trans}Each Anosov element of a TNS totally Anosov action is transitive.
\end{lemma}
\begin{proof}By TNS condition alll non trivial elements $\al$ are Anosov and preserving a fully supported finite measure. Then the non-wandering set for each Anosov element is the whole manifold. It is known that this implies the Anosov diffeomorphism is transitive, cf. \cite{KH} Corollary 18.3.5. 
\end{proof}

\subsection{Standard form for the derivative cocycle of a TNS totally Anosov action}\label{sec: gl rigid}
Throughout the paper, smoothness of diffeomorphisms, actions, and manifolds
is assumed to be $C^\infty$, even though all definitions and some results (especially the key Theorem \ref{main: standard form intro}) can
be formulated in lower regularity. Suppose now action $\al$ is assumed to be an TNS totally Anosov $\ZZ^k$ action on a compact manifold $M$. Then Theorem \ref{main: standard form intro} could be restated as the following, there exist $n\in \ZZ^+$ and a finite cover $\bar M$ of $M$ such that 
\begin{maintheorem}\label{thm: osl spltt cont}
\begin{enumerate}
\item (\emph{Independence of the choice of measure}) The Lyapunov exponents of $D\bal$ are the same for any $\bal-$invariant ergodic measure $\mu$. Here $\bal$ is the lifting of $\al|_{n\cdot \ZZ^k}$ on $\bar M$.
\item (\emph{H\"older continuity of Oseledec splitting}) The cocycle $D\bal|_{n\cdot \ZZ^k}$ preserves a H\"older continuous Oseledec splitting on $T\bar M$:
\begin{equation}\label{eqn: hold osel spl }
T\bar M=\bigoplus \bar E_\chi
\end{equation}
\item (\emph{Polynomial deviation}) Within each Oseledec space $\bar E_\chi$, $D\bal$ has the \textbf{polynomial deviation} in the sense that there exist $C>1, m\in \ZZ^+$ such that for any unit vector $v\in \bar E_\chi $, any $b\in \ZZ^k$,  
\begin{equation}\label{eqn: ply dev within Osl space}
C^{-1}(\|b\|+1)^{-m}e^{\chi(b)}\leq \|D\bal(b)\cdot v\|\leq C(\|b\|+1)^{m}e^{\chi(b)}
\end{equation}
\item (\emph{H\"older continuous Zimmer amenable reduction})Within each Oseledec space $\bar E_\chi$, there exist a $D\bal|_{n\cdot \ZZ^k}-$invariant H\"older continuous flag $$\{0\}=\cE_{\chi,0}\subsetneqq \cE_{\chi,1}\subsetneqq\cdots \subsetneqq \cE_{\chi, l(\chi)}=\bar E_\chi$$
and a family of H\"older continuous metrics $\|\|_{\chi, i}$ within the bundles $\cE_{\chi,i}/\cE_{\chi, i+1}$ such that for any $v
\in \cE_{\chi,i}/\cE_{\chi, i+1}$, any $b\in \ZZ^k$,
\begin{equation}\label{eqn: unif exp gr}
\|D\bal(nb)\cdot v\|_{\chi, i}=e^{\chi(nb)}\cdot \|v\|_{\chi,i}
\end{equation}
\end{enumerate} 
\end{maintheorem}

As a consequence, we can prove Theorem \ref{coro: conserv intro}.

\begin{proof}[Proof of Theorem \ref{coro: conserv intro}]We consider the finite cover $\bar M$, the lift $\bal$ and the positive integer $n$ we got from Theorem \ref{thm: osl spltt cont}. By the part on Zimmer reduction in Theorem \ref{thm: osl spltt cont}, we can  construct a H\"older continuous volume form $\nu_\chi$ on $\bar E_\chi$ by taking the wedge product of the volume form induced by $\|\|_{\chi,i}$ within each $\cE_{\chi,i}/\cE_{\chi, i-1}$. Then we set $$\nu:=\wedge_{\chi} \nu_\chi$$
By \eqref{eqn: unif exp gr}, for any element $b\in \ZZ^k$  $\bal(nb)$ we have $D\bal(nb)_\ast \nu=e^{\sum_{\chi}n\chi(b)}\cdot \nu$. As the proof of Lemma \ref{lemma: elem trans} we know without loss of generality we could assume $\bal(nb)$ is a transitive Anosov diffeomorphism on $\bar M$.


Suppose that $\nu=\phi_0(x)\cdot\nu_0$, where $\nu_0$ is an arbitrary smooth volume on $\bar M$, then $\phi_0$ is a continuous solution of the cohomological equation 
\begin{equation}\label{eqn: coh eqn vol}\phi(\bal(nb)\cdot x)^{-1}\cdot \phi(x)=e^{-\sum_{\chi}\chi(nb)}\cdot \det(D\bal(nb) , \nu_0)
\end{equation}
since the right hand side of \eqref{eqn: coh eqn vol} is smooth and $\bal(nb)$ is a smooth transitive Anosov diffeomorphism, and $\phi_0(x)$ is bounded from below,  by Livsic theorem (see \cite{LMM}) the continuous solution $\phi_0$ is in fact smooth. Therefore $\bal|_{n\ZZ^k}$ preserves a smooth volume form $\nu$ on $\bar M$. Then $\nu$ induces a probability measure $p_{\nu}$ on $\bar M$.

Let $p_{\nu'}:=\Pi_\ast(p_\nu)$ to be the probability measure on $M$, where $\Pi: \bar M \to M$ is the covering map. Then obviously $p_{\nu'}$ is induced from some smooth volume form $\nu'$ on $M$. Moreover since for any $b\in \ZZ^k, ~~ \Pi\circ\bal(b)=\al(b)\circ\Pi$, then $p_{\nu'}$ as well as $\nu'$ is $\al|_{n\cdot \ZZ^k}-$invariant.

To get an $\al|_{\ZZ^k}-$invariant measure (not only $\al|_{n\ZZ^k}-$invariant), we consider $$p_{\nu''}:=\frac{1}{N^k}\sum_{0\leq c_i<N}\al(\sum_{i=1}^k c_i e_i)_\ast p_{\nu'}$$
where $\{e_i,i=1,\dots k\}$ is the canonical basis of $\ZZ^k$. Since for any $b$, $p_{\nu'}$ is $\al(nb)-$invariant, then by our construction $p_{\nu''}$ is indeed $\al-$invariant. Obviously $p_{\nu''}$ is induced from a smooth volume form $\nu''$, hence $\al$ is a volume preserving action.\end{proof}
\color{black}

\subsection{Global rigidity for TNS totally Anosov $\ZZ^k-$actions}\label{sec: main thm rig chp 2}

Recall that a \emph{nilmanifold} is a quotient $N/\Gamma$, where $N$ is a simply connected
nilpotent Lie group and $\Gamma$ is a discrete subgroup. An \emph{infranilmanifold}
$M$ is a manifold finitely covered by a nilmanifold. We call a diffeomorphism of $N$ affine if it is a composition of
an automorphism of $N$ with a left translation by an element of $N$.  Diffeomorphisms of $M$ covered
by affine diffeomorphisms of N are again called affine. The Franks-Manning conjugacy
theorem generalizes to infra-nilmanifolds: Suppose $M'$ is a smooth manifold
homeomorphic with an infra-nilmanifold. Then every Anosov diffeomorphism of
$M'$ is conjugate to an affine diffeomorphism of $M$ by a homeomorphism.

Let action $\al$ be TNS totally Anosov. As mentioned in the introduction, we have the following smooth rigidity results for TNS totally Anosov $\ZZ^k-$actions. 

As in \cite{KS06}, we say two foliations $\W^1$ and $\W^2$ of a manifold are \textit{topologically jointly
integrable} if there is a topological foliation $\W$ such that for any $x$ the map
$$\phi: \W^1(x)\times\W^2(x)\to\W(x): \phi (y,z)\to \W^1(y)\cap\W^2(z)$$
is a well-defined local homeomorphism, i.e. the foliations $\W^1$ and $\W^2$
give a local product structure on the leaves of $\W$.

\begin{maintheorem}\label{main toral case}  If any two coarse Lyapunov foliations of $\al$ are topologically jointly integrable, then $\al$ is smoothly conjugate to a $\ZZ^k-$action by affine automorphisms on an infranilmanifold.
\end{maintheorem}

The following higher rank assumption for Anosov action is considered in \cite{KS07}.
\begin{definition}\label{res-free}$\al$ is called \emph{resonance-free} if there are no Lyapunov exponents $\chi_1, \chi_2,\chi_3$ such that $\chi_1-\chi_2$ is proportional to $\chi_3$. 
\end{definition}
\emph{A priori}, the resonance-free condition depends on the choice of the ergodic invariant measure $\mu$, while the TNS condition does not. However, by Theorem \ref{thm: osl spltt cont} we know under the TNS and totally Anosov assumption,  resonance-free condition actually does not depend on the choice of $\mu$ either.

\begin{maintheorem}\label{main resonance free case} If $\al$ is resonance free, then a finite cover of $\al$ is smoothly conjugate to a $\ZZ^k-$action by affine automorphisms on a torus. \end{maintheorem}

Now we consider a \textit{narrow band type} condition for Lyapunov exponents. We say $\al$ is \textit{Lyapunov pinching} if there is a nontrivial element $\al(a),a\in \ZZ^k$ such that the top and bottom Lyapunov exponents $\lambda^{u(s)}_\pm(a)$ within stable and unstable distributions of $\al(a)$ satisfy one of the following inequalities:
$$1+\frac{\lambda^u_-(a)}{\lambda^u_+(a)}>\frac{\lambda^s_-(a)}{\lambda^s_+(a)}, 1+\frac{\lambda^s_+(a)}{\lambda^s_-(a)}> \frac{\lambda^u_+(a)}{\lambda^u_-(a)}$$
A uniform version of our condition has been considered in \cite{B77, B78, B80, BM}.
\begin{maintheorem}\label{main Lyapunov pinching}If $\al$ is Lyapunov pinching then $\al$ is smoothly conjugate to a $\ZZ^k-$action by affine automorphisms on an infranilmanifold.
\end{maintheorem}

An immediate corollary of Theorem \ref{main Lyapunov pinching} is the following result, which can be viewed as a higher-rank, higher dimensional and smooth analogue of Newhouse's result in \cite{N} for topological rigidity of codimension-one Anosov diffeomorphism. 
\begin{coro}\label{coro: one LE} If there is an element $\al(a)$ such that it has only one Lyapunov exponent on its stable or unstable distribution, then $\al$ is smoothly conjugate to a $\ZZ^k-$action by affine automorphisms of an infranilmanifold. 
\end{coro} 

\begin{proof}If there is an element $a$ such that $\al(a)$ has only one Lyapunov exponent on its stable or unstable distribution, then by polynomial properties in Theorem \ref{thm: osl spltt cont} we know a finite cover of $\al$ is Lyapunov pinching. But since the covering map could be chosen as a local isometry, therefore $\al$ itself is Lyapunov pinching as well. Then Corollary \ref{coro: one LE} is a direct consequence of Theorem \ref{main Lyapunov pinching}.
\end{proof}


\section{Preliminaries}\label{sec: prel}
\subsection{Anosov $\RR^k-$actions and suspensions.}
An $\RR^k-$action $\al$  on a smooth manifold is Anosov if it has at least one element which has a hyperbolic splitting into stable and unstable distributions, transversally to the orbit distribution (which is tangent to the orbit foliation of $\al$). Any such element we will call an \emph{Anosov element} of the action $\al$. In particular an $\RR^k-$action is called \emph{totally Anosov} if the set of Anosov elements $\mathcal{A}$ is dense in $\RR^k$.

For a $\ZZ^k-$action $\al$ on a manifold $M$ there is an associated canonical $\RR^k-$action $\tilde{\al}$ on $\tM$ given by standard suspension construction.
\begin{eqnarray}
\tM:=\RR^k\times M/{\sim},&&\text{where } (t,x)\sim (t-n , \al(n),x), n\in \ZZ \\
\tal:\RR^k\times \tM\to \tM,&& \tal(t)\cdot(s,x)=(s+t,x)
\end{eqnarray}
We call $\tal$ is the \emph{suspension} of $\al$. In the following proposition we summarize several important properties for the suspension $\tal$ of a $\ZZ^k$ Anosov action $\al$, for the proof and more details see \cite{KS07}, \cite{KalSpa}.
\begin{prop}\label{prop: properties suspen}Let $\al$ be a smooth totally Anosov $\ZZ^k-$action on $M$. Then we have
\begin{enumerate}
\item The restriction of $\tal$ on $\ZZ^k$ preserves each transversal copy of $M$ in $\tM$, i.e. $$t\in \RR^k, (\{t\}+\ZZ^k)\times M/\sim~~ \subset \tM$$ 
\item Any $\al-$invariant measure on $M$ lifts to a unique $\tal-$invariant measure on $\tM$.
\item There is an $\tal-$invariant H\"older continuous coarse Lyapunov splitting 
\begin{equation}\label{eqn: coarse splt suspension}T\tM=T\O\oplus TM=T\O\oplus E^i
\end{equation}where $M$ here is identified with the transversal copy of $M$ in $\tM$ mentioned above, $T\O$ is the tangent distribution of the $\RR^k-$orbits of $\tal$ and $E^i$ are the finiest nontrivial intersections of the stable distributions of the Anosov elements of the actions. Each $E^i$ is tangent to a H\"older foliation $W^i$ with uniformly $C^\infty-$leaves.
\item The Lyapunov hyperplanes and the picture of Weyl chambers are the same for all $\tal-$invariant ergodic probability measures. The set $\A$ of Anosov elements for $\tal$ is the union of the Weyl chambers in $\RR^k$. In particular $\tal$ is totally Anosov. 
\item For any $\tal-$invariant ergodic probability measure, the coarse Lyapunov splitting coincides on the set of full measure with the splitting \eqref{eqn: coarse splt suspension}. Moreover, for each Lyapunov exponent $\chi$, 
$$E^{\chi}(p)=\bigcap_{a\in \A, \chi(a)<0}E^s_a(p)=\bigoplus_{\chi'=c\chi, c>0}E_{\chi'}(p)$$
\item Under the TNS condition, almost every element
of every Lyapunov hyperplane is transitive on $\tM$.
\end{enumerate}
\end{prop}

\subsection{Foliation's \textit{holonomies} and \textit{Holonomies} of fibered-bunched cocycles}\label{sec: holo HOLO}
In this paper we consider two different classes of  maps both are all called \emph{holonomy} in the previous work. To distinguish them we will use the words ``holonomy" for foliation's holonomies and ``Holonomy" for the one induced by fibered-bunched cocycles. 

Suppose that there are two transeverse $F_1,F_2$ in an $n-$dimensional manifold $M$ where $\dim F_1=k, \dim F_2=n-k$. Then for two $F_1-$leaves $F_1(x), F_1(y)$ close enough, we could define the holonomy map $h: F_1(x)\to F_1(y)$, for any $z$ in the local leaf of $F_1(x)$, $h(z)$ is the unique intersection of the local leaf of $F_2(z)$ with $F_1(y)$. 

Now we consider the \textit{Holonomies} induced by fiber bunched cocycles. For details cf. \cite{ASV} \cite{KS13} and references therein. Recall that a $C^1-$diffeomorphism $f$ on a compact smooth manifold $M$ is called partially hyperbolic if there is a $Df-$invariant splitting $E^s\oplus E^c\oplus E^u $ of $TM$ such that there exists $k>1$ such that for any $x\in M$ and any choice of unit vectors $v^s\in E^s$, $v^c\in E^c$, $v^u\in E^u$, $$\|Df^k(v^s)\|<1<\|Df^k(v^u)\|$$
$$ \|Df^k(v^s)\|<\|Df^k(v^c)\|<\|Df^k(v^u)\|$$
If $E^c$ is trivial then  $f$ is Anosov. \footnote{In general we could assume $k$ in the definition of partial hyperbolicity to be $1$ by considering Lyapunov metrics.}Suppose $f:M\to M$ is a partially hyperbolic diffeomorphism on a compact manifold $M$. Then we could pick a continuous function $\nu<1$ such that  for any unit vector $v\in E^s_x$, 
\begin{equation}\label{eqn: cont rate}
\|Df(x)\cdot v\|<\nu(x)
\end{equation}
$P : E\to M$ is a finite dimensional $\beta-$H\"older vector bundle over $M$; and 
$F : E\to E$ is a $\beta-$H\"older linear cocycle over $f$ (i.e. $F_x$ $\beta-$H\"older continuously depends on $x$). $F$ is called
\textit{fiber bunched} (over $W^s(f)$) if for some $\beta-$H\"older norm on $E$,
\begin{equation}\label{eqn: f bunch def}
\|F(x)\|\cdot \|F(x)^{-1}\|\cdot \nu(x)^\beta<1
\end{equation}
for all $x$ in $M$. Similarly we could define cocycle fiber bunched over unstable manifolds of a partially hyperbolic diffeomorphism. In section \ref{sec:poly-deviation} we will meet some cocycles which is fiber-bunched on $W^s$ but not on $W^u$. One of the key problems for this paper is to describe the dynamics of them.

\begin{definition}The stable Holonomy for $F : E\to E$ defined above is a continuous
map $H^s:(x,y)\to H^s_{xy}$ where $x\in M$, $y\in W^s(x)$ such that
\begin{enumerate}
\item $H^s_{xy}$ is a linear map from $E_x$ to $E_y$.
\item $H^s_{xx}=\mathrm{id}$ and $H^{s}_{yz}\circ H^s_{xy}=H^s_{xz}$.
\item $H^s_{xy} = (F^n_y)^{−1} \circ H^s_{f^nx f^ny}
\circ F^n_x$ for all $n\in N$.
\end{enumerate}
The unstable Holonomy for $F$ could be defined similarly.
\end{definition}
The following proposition lists some important properties of Holonomies of fiber-bunched cocycles. For $x,y\in M$ two nearby points we let $I_{xy} : E_x\to E_y$ be a
linear identification which is $\beta-$H\"older close to the identity.
\begin{prop}\label{prop: properties holonomies}[Proposition 4.2 \cite{KS13}] Suppose that the cocycle $F$ is fiber bunched. Then there exists $C > 0$ such that for any $x\in M$ and $y\in W^s_{loc}(x)$,
\begin{itemize}
\item[a.] $\|(F^n_y)^{−1}\circ I_{f^nx f^ny}\circ  F^n_x − I_{xy}\| ≤ C\d(x, y)^\beta$
for every $n\geq 0$;
\item[b.] $H^s_{xy} = \lim_{n\to \infty}(F^n_y)^{−1}\circ I_{f^nxf^ny}\circ F^n_x$
satisfies (1), (2), (3) in definition of Holonomies and
$$\textbf{(4).}~~~~~\|H^s_{xy} − I_{xy}\| ≤ C\d(x, y)^\beta$$
\item[c.] The stable holonomy satisfying (iv) is unique.
\end{itemize}
$H^s_{xy}$ can be extended to any $y\in W^s(x)$ using (3). Similar properties also hold for unstable Holonomy.
\end{prop}
\begin{remark}In section \ref{sec: non-uni coc Holo} we will meet the \textit{non-uniform Holonomy} induced by \textit{non-uniformly} fiber-bunched cocycle which is useful in the proof of Theorem \ref{thm: osl spltt cont}. 
\end{remark}

\section{Subexponential growth for elements on a Lyapunov hyperplane}\label{section: uniform estimate}
In this section we prove a useful subexponential growth estimates for elements on a Lyapunov hyperplane, which generalize Proposition 4.2 in \cite{FKS1} to general compact manifolds. As in section \ref{sec: prel},  $\al$ is always to be a TNS totally Anosov $\ZZ^k$ action on $M$, and $\tilde \al$ is the suspension action of $\al$.

\begin{lemma}\label{lemma:sub exp grwth on plane}
Let $L\subset \RR^k$ be a Lyapunov hyperplane and $E$ be the corresponding coarse Lyapunov distribution for $\tal$. For any $\epsilon>0, b\in L$, there exist $C >0$ such that for any $n>0$ and any unit vector $v\in E$,
\begin{equation}\label{eqn: sub exp grwth}
C^{-1}e^{-\epsilon n}\leq \|D(\tal(nb))\cdot v\|\leq Ce^{\epsilon n}
\end{equation}
\end{lemma}
\begin{proof}In the proof we abbreviate $\tal(b)$ to $b$. Firstly we claim that for any $b\in L$ any $b-$invariant ergodic measure $\mu$ on $\tM$, for the cocycle $Db$ we have 
\begin{equation}\label{eqn: vanished LE on L}
\lambda^+_E(Db, \mu)=\lambda^-_E(Db, \mu)=0
\end{equation}
where $\lambda^\pm_E$ are the top and bottom Lyapunov exponents for the derivative cocycle restricted on $E$. Without loss of  generality we only need to prove that $\lambda^+_E(Db,\mu)\leq 0$ (replace $b$ by $-b$ we get $\lambda^-_E(Db,\mu)\geq 0$). 

Suppose $\lambda^+_E(Db,\mu)>\epsilon>0$, we first consider the case where the top Lypaunov exponent $\lambda^+_E(Db,\mu)$ of the cocycle $(Db,\mu)$ is simple. Consider the projective bundle $\PP E$ of $E$ over $\tM$ and the projective action $\PP (Db|_E)$ of $Db|_E$ on $\PP E$. Then by Oseledec theorem, the one-dimensional Oseledec space $V_1$  within $E$ corresponding to the top Lyapunov exponent of cocycle $(Db|_E, \mu)$ induces a $\mu-a.e.-$defined $\PP (Db|_E)-$invariant measurable section $ \PP(V_1(x))$ of $\PP E$. 

Obviously the graph of the section $\{(x, \PP(V_1(x))\}$ is a measurable set in $\PP E$ (see \cite{V14}) and we can define a Borel  probability measure $\mu^{V_1}$ on $\PP E$ as the following: for all measurable set $U\subset \PP E$
$$\mu^{V_1}(U):= \mu (\pi(U\cap \{(x, \PP(V_1(x))\}))$$
where $\pi$ is the canonical projection from $\PP E$ to $\tM$.

For any $\bar v\in \PP E, c\in \RR^k$, we define $$\ln\|Dc\cdot \bar v\|:=\frac{\ln\|Dc\cdot v\|}{\|v\|}, \text{ where } v\in E-\{0\} \text{ and } v \text{ is in the class } \bar v$$
Obviously $\ln\|Dc\cdot \bar v\|$ does not depend on the choice of $v$, moreover it is an additive cocycle in the sense that for any $c_1,c_2\in \RR^k, \bar v\in \PP E$ we have 
\begin{equation}\label{eqn: cocycle derivative}
\ln\|D(c_1+c_2)\cdot \bar v\|=\ln \|Dc_2\cdot (\PP Dc_1\cdot \bar v)\|+\ln \|Dc_2\cdot \bar v\|
\end{equation}
In summary, $\ln \|Dc\|$ is a well-defined function on $\PP E$. Now we consider the set $C$ formed by all $\PP(Db)-$invariant Borel probability measure $\nu$ of $\PP E$ such that $$(\ln \|Db\|, \nu):=\int_{\PP(E)}\ln\|Db\|d\nu \geq \epsilon$$
By our assumption on $\mu^{V_1}$, we have $\mu^{V_1}\in C$ hence $C$ is non-empty, moreover we have the following
\begin{lemma}\label{lemma: cmmn fxd pt}There is a $\PP(D\tal)_\ast-$common fixed point in $C$.
\end{lemma}
\begin{proof}
It is easy to check $C$ is a compact convex subset in the dual of $C^0(\PP E)$. Moreover for any $c\in \RR^k, \nu_1, \nu_2\in C, t\in [0,1]$, $$\PP( Dc)_\ast (t\nu_1+(1-t)\nu_2)=t\PP(Dc)_\ast \nu_1+(1-t)\PP(Dc)_\ast\nu_2 $$
Moreover, for any $c\in \RR^k$, any  $\PP(Db)-$invariant Borel probability measure $\nu$ of $\PP E$
\begin{eqnarray*}
&&(\ln\|Db\|, ~~\PP(Dc)_\ast\nu)\\
&=&\int_{\PP E} \ln\|Db\cdot (\frac{Dc\cdot v}{\|Dc\cdot v\|})\|d\nu(\bar v), ~~\text{$v$ is arbitrary unit vector in $\bar v$}\\
&=&\int_{\PP E} \ln\|D(c+b)\cdot v\|-\ln\|Dc\cdot v\|d\nu(\bar v)\\
&=&\int_{\PP E} \ln\|Dc\cdot (\PP Db\cdot \bar v)\|+\ln\|Db\cdot \bar v\|-\ln\|Dc\cdot \bar v\|d\nu(\bar v)~~(\text{by \eqref{eqn: cocycle derivative}})\\
&=&(\ln\|Dc\|, \PP(Db)_\ast \nu)+(\ln\|Db\|, \nu)-(\ln \|Dc\|, \nu )\\
&=&(\ln\|Db\|, \nu)~~(\text{since $\nu$ is $\PP(D\al)_\ast-$invariant})
\end{eqnarray*}
Therefore $C$ is $\PP(D\tal)_\ast-$invariant. Then by Markov–Kakutani fixed-point theorem  (that is:  a commuting family of continuous affine self-mappings of a compact convex subset in a locally convex topological vector space has a common fixed point, see \cite{RS} page 152.), there exists a $\PP(D\tal)_\ast-$common fixed point in $C$. \end{proof} 
As a result of Lemma \ref{lemma: cmmn fxd pt}, there exists a $\PP(D\tal)_\ast-$ invariant measure $\nu_0$ on $\PP E$ such that $(\ln \|Db\|, \nu_0)\geq \epsilon$, now we pick a $b'$ close to $b$ such that $b'$ is the negative Lyapunov half space corresponding to $L$. By smoothness of $\al$ we can choose $b'$ such that $$|\ln\|Db\|-\ln\|Db'\||<\frac{1}{2}\epsilon$$Therefore by our choice of $\nu_0$ we have 
\begin{equation}\label{eqn: closeness b'}
(\ln\|Db'\|, \nu_0)>\frac{1}{2}\epsilon
\end{equation}
But since $b'\notin L$ and is close enough to $b$, without loss of generality we could further assume that $b'$ is regular, i.e. does not lie in any Lyapunov hyperplane. Then by the totally Anosov  assumption on $\al$ and Proposition \ref{prop: properties suspen}, $b'$ is an Anosov element and $E\subset E^s_{b'}$. Therefore there exists $n=n(b')\in \ZZ^+$ such that $\|D(nb')|_E\|<1$. As a result, by $D\tal-$invariance of $\nu_0$ we have 
\begin{eqnarray*}
(\ln\|Db'\|, \nu_0)&=&(\ln\|Db'\|, \PP(Db')_\ast\nu_0)\\
&=&(\ln\|Db'\|, \PP(D(2b')_\ast\nu_0)\\
&=&\cdots\\
&=&(\ln\|Db'\|, \PP(D(n-1)b')_\ast\nu_0)\\
&=&\frac{1}{n}\sum_{k=0}^{n-1}(\ln\|Db'\|, \PP(Dkb')_\ast\nu_0)\\
&=&\frac{1}{n}(\ln \|D(nb')\|,\nu_0)~~(\text{by \eqref{eqn: cocycle derivative}})\\
&<&0~~(\text{since }\|D(nb')|_E\|<1)
\end{eqnarray*}
which contradicts  \eqref{eqn: closeness b'}.

In summary, if $\lambda^+_E(Db,\mu)$ is the simple top Lypaunov exponent then $\lambda^+_E(Db,\mu)>0$ cannot hold. For the case that the $k-$th and $(k+1)-$th Lyapunov exponents of cocycle $(Db|_E, \mu)$ are distinct, by considering the projective bundle  $\PP(\bigwedge^k(E))$ and by mimicking  the proof above, we can show that $\lambda^+_E(Db,\mu)\leq 0$. Therefore we prove that for any $b\in L$ and for any $b-$invariant measure $\mu$, \eqref{eqn: vanished LE on L} holds.

Now the proof of Lemma \ref{lemma:sub exp grwth on plane} is essentially  the same as section 4 of \cite{K11}, for completeness we give an outline here. We need the following lemma:
\begin{lemma}[cf. \cite{H07}, Prop 3.4.] \label{lemma: subadd} Suppose $f$ is a continuous map of a compact metric space $X$ and $a_n: X\to \RR, n\geq 0$ is a sequence of continuous functions such that 
\begin{equation}\label{eqn: 1. lemma H07}
a_{n+k}(x)\leq a_n(f^k(x))+a_k(x)~~\text{for every } x\in X, n,k\geq 0
\end{equation}
If there is a sequence of continuous functions $b_n: X\to \RR, n\geq 0$ such that 
\begin{equation}\label{eqn: 2. lemma H07}
a_n(x)\leq a_n(f^k(x))+a_k(x)+b_k(f^n(x))~~\text{for every } x\in X, n,k\geq 0
\end{equation}
and $\inf_{n}(\frac{1}{n}\int_Xa_nd\mu)<0$ for every ergodic $f-$invariant measure $\mu$, then there is $N\geq 0$ such that $a_N(x)<0$ for every $x\in X$.  
\end{lemma}

Now we fix $b\in L$ and take $\epsilon>0$ arbitrary small. Let $a_n$ be $\log\|D(nb)|_E\|-n\epsilon$, then by Kingman subadditive ergodic Theorem \cite{K68} and equality  \eqref{eqn: vanished LE on L}, we have for any $b-$invariant measure $\mu$ on $M$, $\mu-a.e. x\in M$, $$\inf_{n}\frac{1}{n}\int_X a_nd\mu=\lim_{n\to \infty}\frac{1}{n}a_n(x)=\lambda^+_E(Db,\mu)-\epsilon<0$$ 
Then $a_n, n>0$ satisfy \eqref{eqn: 1. lemma H07}. Let $b_n$ be $\log\|D(-nb)|_E\|$. It is easy to see that $a_n, b_n$ satisfy \eqref{eqn: 2. lemma H07}. Applying Lemma \ref{lemma: subadd} to $a_n, n>0$, we have that for any $\epsilon>0$ there exists $N_\epsilon>0$ such that $a_{N_\epsilon}<0$, i.e. $$\|D(N_\epsilon b)|_E\|<e^{\epsilon N_\epsilon}.$$ This easily implies the right hand side of \eqref{eqn: sub exp grwth}. In a completely similar way we get the left hand side of \eqref{eqn: sub exp grwth}.
\end{proof}

\section{A specific splitting}\label{chap: spltt}
In this section we fix a Lyapunov hyperplane $L\subset \RR^n$ for $\al$ and denote by $E$ and $W$ the corresponding coarse Lyapunov distribution and foliation for $\al$ on $M$. We will establish smoothness of holonomies along certain stable leaves for some action elements between leaves of $W$, by using the estimates from section \ref{section: uniform estimate}. More precisely, 
\begin{prop}\label{prop: E12 splt}
There exist $D\al-$invariant distributions $E_1, E_2$ and Anosov elements $c_i\in \ZZ^k, i=1,\dots, 4$ such that
\begin{enumerate}
\item $TM=E_1\oplus E\oplus E_2$.
\item $E_1=E^s_{c_3}, E_2=E^s_{c_4}, E\oplus E_1=E^s_{c_1}, E\oplus E_2=E^s_{c_2}$. The distributions $E_1, E_2, E\oplus E_1, E\oplus E_2$ are tangent to $\al-$invariant foliations which we denote by $W_1, W_2, W\oplus W_1, W\oplus W_2$ respectively.
\item  For $i=1,2$, $c_i$ contracts $W_i$ faster than it does $W$, i.e.
\begin{equation}\label{eqn: contract faster}
\|D\al(c_i)|_{E_i}\|<\|D\al(-c_i)|_E\|^{-1}\leq \|D\al(c_i)|_E\|<1, i=1,2.
\end{equation}
\item The H\"older foliations $W, W_1, W_2, W\oplus W_1, W\oplus W_2$ have uniformly smooth leaves. The holonomies along $W_i$ between leaves of $W$ are smooth. Each $W_i$ is a smooth subfoliation within $W\oplus W_i$, $i=1,2$. 

\item In addition, for any given $\beta>0$, $c_i, i=1,\dots, 4$ can be chosen to satisfy the estimate:
\begin{equation}\label{eqn:FiberBu}
\|D\al(c_i)|_{E}\|\cdot \|D\al(-c_i)|_{E}\|\cdot \|D\al(c_i)|_{E_j}\|^\beta<1,~~ \text{where } 2| i-j.
\end{equation}
\end{enumerate}
\end{prop}
\begin{proof}
We take a generic plane $P$ in $\RR^k$ which intersects different Lyapunov hyperplanes by different lines. We can order these oriented lines and corresponding negative Lyapunov half-spaces cyclically $L=L_1, \cdots, L_n$. The TNS assumption implies that different negative Lyapunov half-spaces correspond to different Lyapunov hyperplanes. Let $m$ be the index such that $-L_1$ is between $L_m$ and $L_{m+1}$. There are two Weyl chambers $\C_1, \C_2$ in the negative Lyapunov half-space $L_1^-$ whose intersections with the plane $P$ border $L_1$. Similarly there are two Weyl chambers $\C_3, \C_4$ across $L_1$ in the positive Lyapunov half-space $L_1^+$. 

By the totally Anosov assumption on $\al$ (see Proposition \ref{prop: properties suspen}), all elements in $\C_i, i=1,\dots, 4$ are Anosov. Then for any $c_i\in \C_i\cap \ZZ^k, i=1,\dots 4$, and denote by
\begin{eqnarray}\label{eqn: def c34}
E_1:=E^s_{c_3}, ~~E_2:=E^s_{c_4}
\end{eqnarray} 
Then reorder $\C_1, \C_2$ if necessary, we have (cf. section 5 of \cite{FKS1})
\begin{equation}\label{eqn: def c12}
E\oplus E_1=E^s_{c_1},~~E\oplus E_2=E^s_{c_2} 
\end{equation}
Then we have $TM=E_1\oplus E\oplus  E_2$ and $E_1, E_2, E\oplus E_1, E\oplus E_2$ are tangent to invariant foliations $W_1, W_2, W\oplus W_1, W\oplus W_2$ respectively. Since stable foliation of Anosov diffeomorphism always H\"older continuous and with smooth leaves. Then we prove (1),(2), and first half of (4) of Proposition \ref{prop: E12 splt}.

It is well-known that a fast part of a stable foliation is $C^\infty$ inside a stable leaf (cf. \cite{FKS1}, \cite{KS07}). Therefore to prove the rest of (4) we only need to prove (3). We take $c_1'\in \C_1\cap P$ and close to a unit vector $b\in L\cap P$. We consider the path $b^1(t)=(1-t)b+tc_1'$, by our assumption $\{b^1(t), t\in (0,1)\}$ is contained in $\C_1\cap P$. We claim that for $t>0$ small enough, $n>0$ sufficiently large, $b_1=nb^1(t)$ satisfies that
\begin{equation}\label{eqn: b1 faster}
\|D\tal(b_1)|_{E_1}\|<\|D\tal(-b_1)|_E\|^{-1}\leq \|D\tal(b_1)|_E\|<1
\end{equation}

Since $b$ is not in any other Lyapunov hyperplane, $b$ uniformly contracts $E_1$. Therefore there exist $\theta, c, $ such if $t$ small enough, for any $n$ we have 
\begin{equation}\label{eqn: nb1t est E1}
\|D\tal(nb^1(t))|_{E_1}\|<ce^{-\theta n}
\end{equation}

On the other hand, by smoothness of $\tal$, there exists $c'>0$ such that for $t>0$ small, we have 
\begin{equation}\label{eqn: estm b1t-b}
e^{-c't}<\|D\tal(b^1(t)-b)|_E\|<e^{c't}, ~~ e^{-c't}<\|D\tal(b- b^1(t))|_E\|<e^{c't}
\end{equation}

We fix a $t$ such that $c't\ll \theta$ and \eqref{eqn: nb1t est E1} holds, where $\theta$ is defined in \eqref{eqn: nb1t est E1}. Since $b^1(t)$ is an Anosov element, then for $n$ large enough, $b_1:=nb^1(t)$ uniformly contracts $E$, i.e. 
\begin{equation}\label{eqn: b1 contracts E}
\|D\tal(b_1)|_{E}\|<1
\end{equation}
Notice that for any small positive number $\epsilon$ such that $\epsilon\ll \theta$ by \eqref{eqn: sub exp grwth} there exists $C(\epsilon)>0$ such that 
\begin{equation}\label{eqn: sub exp -nb}
\|D\tal(-nb)|_E\|<C(\epsilon) e^{\epsilon n}
\end{equation}

Then for $n$ large enough 
\begin{eqnarray}\nonumber
\|D\tal(-b_1)|_E\|&\leq & \|D\tal(-nb)|_E\|\cdot \|D\tal(n(b-b^1(t)))|_E\|\\ \nonumber
&\leq & C(\epsilon)e^{\epsilon n}\cdot e^{c'nt}~~(\text{by \eqref{eqn: estm b1t-b}, \eqref{eqn: sub exp -nb}})\\\label{eqn: est -b1 E}
&\leq &c^{-1}e^{\theta n}~~(\text{since }\epsilon, c't\ll \theta \text{ and $n$ large enough})
\end{eqnarray}
Then by \eqref{eqn: nb1t est E1} we have that for $n$ large, $b_1=nb^1(t)$ satisfies that $$\|D\tal(b_1)|_{E_1}\|<\|D\tal(-b_1)|_E\|^{-1}$$
Combining with \eqref{eqn: b1 contracts E} we prove that  $b_1=nb^1(t)\in \RR^k$ satisfies \eqref{eqn: b1 faster} if $n$ large. Replace  $c_1$ by the nearest integer point in $ \C_1\cap \ZZ^k$ to $b_1$. If $n$ sufficiently large, the $c_1$ we choose satisfies \eqref{eqn: contract faster}.
Similarly we can choose $c_2$ corresponding to $E_2$, therefore we obtain (3) of Proposition \ref{prop: E12 splt}, hence (4).

Proving  (5) is essentially the same as (3). Namely,  take $c_3'\in \C_3\cap P$ and close to a unit vector $b\in L\cap P$. We consider the path $b^3(t)=(1-t)b+tc_3'$, by our assumption $\{b^3(t), t\in (0,1)\}$ is contained in $\C_3\cap P$. We claim that for $t>0$ small enough, $n>0$ sufficiently large, $b_3=nb^3(t)$ satisfies 
\begin{equation}\label{eqn: b bunching}
\|D\tal(b_3)|_{E}\|\cdot \|D\tal(-b_3)|_{E}\|\cdot \|D\tal(b_3)|_{E_1}\|^\beta <1
\end{equation}
As in the proof of (3), by smoothness of $\tal$ and subexponetial estimate in section \ref{section: uniform estimate}, for any $\epsilon>0$, if $t>0$ sufficiently small  then $\|D\tal(nb^3(t))|_{E}\|\cdot \|D\tal(-nb^3(t))|_{E}\|=o(e^{\epsilon n})$ when $n\to \infty$. But for $t>0$ sufficiently small $\|D\tal(nb^3(t))|_{E_1}\|$ tends to $0$ uniformly exponential fast and the speed does not depend on how small $t$ is. Therefore as in  the proof of (3), we can easily choose $b^3$ hence $c_3$ satisfies \eqref{eqn:FiberBu}, and similarly we can choose $c_i, i=1,2,4$ so that they satisfy \eqref{eqn:FiberBu}.
\end{proof}

\section{Proof of Theorem \ref{thm: osl spltt cont}}\label{sec:poly-deviation}
The outline of this section is the following. In sections \ref{sec: BM meas}, \ref{sec: unif bunched, Holo}, \ref{sec: Rue thm}, \ref{sec: non-uni coc Holo}, we recall some important concepts in the theory of cocycles and smooth dynamical systems which we will use later. In the rest of this section we prove our key Theorem \ref{thm: osl spltt cont}. 


\subsection{Bowen-Margulis measure of a transitive Anosov diffeomorphism}\label{sec: BM meas}
Suppose that there are two transeverse $F_1,F_2$ foliations in an $n-$dimensional manifold $M$ where $\dim F_1=k, \dim F_2=n-k$. Denote by $F_i(x,\mathrm{loc}), i=1,2.$ the local leaves of foliations $F_i$. Then for every $x\in M$ there exists small  relative neighborhoods $B^i(x)\subset F_i(x,\mathrm{loc}),~~i=1,2$, such that $$\iota|_{ B^1(x)\times B^2(x)}: (x_1,x_2)\to [x_1,x_2]:=F_1(x_1, \mathrm{loc})\cap F_2(x_2, \mathrm{loc})$$
defines a homomorphism from $ B^1(x)\times B^2(x)$ to some neighborhood $B(x)$ for every $x\in M$. We say the foliation $F_1, F_2$ has a local product structure. A probability measure $m$ on $M$ has a \textit{local product structure} if for every $x$ in the support of $m$, there exist $m^1$ and $m^2$ on $B^1(x)$ and $B^2(x)$ respectively, such that
$$m|_{B(x)} \sim  \iota_\ast(m^1\times m^2) $$ which means that the two measures have the same zero sets. 
For more details see \cite{AV10} and the references therein.

Recall that (the stable and stable foliations of) an Anosov diffeomorphism $f$ on a compact manifold $M$ has a \emph{local product structure}. In this section we will consider some specific $\al-$invariant measures on $M$, i.e. Bowen-Margulis measure of some transitive Anosov elements\footnote{By uniqueness the Bowen-Margulis measure of an Anosov diffeomorphism, the measure $\mu_{BM(\al(b))}$ of an Anosov element $b$ is $\al-$invariant.}. The following lemma is well-known, cf. \cite{Bo}, \cite{Ma}.
\begin{lemma}Suppose $f$ is a transitive Anosov diffeomorphism on a compact manifold $M$. The \emph{Bowen-Margulis} measure $\mu_{BM}(f)$ of $f$ is the unique entropy maximizing $f-$invariant measure. In addition $\mu_{BM}(f)$ has a \emph{local product structure}.
\end{lemma} 


\subsection{Uniformly fiber-bunched cocycle and its Holonomy}\label{sec: unif bunched, Holo}
In this section we continue our discussion for the \textit{Holonomy} of the H\"older continuous linear cocycle over a partially hyperbolic system. 

We consider the H\"older continuous coarse Lyapunov distribution $E, E_1, E_2$, the foliations $W, W_1, W_2$, and the Lyapunov hyperplane $L$ defined in section \ref{chap: spltt}. We denote by $\beta$ the H\"older exponent of $E,E_1,E_2$. By the uniform control proved in section \ref{chap: spltt}, we can pick elements $c_1, c_2, c_3, c_4$ close to $L$ enough so that \eqref{eqn: def c12}, \eqref{eqn: def c34} hold and $D\al(c_i)|_E$ is a fibered bunched cocycle over $c_i$, $i=3,4$, here $E$ can be viewed as the central distribution for partially hyperbolic diffeomorphisms $\al(c_i),i=1,2$. 

Therefore we can consider the stable Holonomy for cocycles $D\al(c_i)|_E$, we denote by $$H^{W_j}_{c_i, x,y}:=H^j_{c_i,x,y}$$ if $W_j\subset W^s(c_i)$. Consider the holonomy map $h^i$ along $W_i$ between two $W$ leaves. By uniqueness of Holonomy in Proposition \ref{prop: properties holonomies}, we have 
\begin{lemma}\label{lemma: Hlmy 2nd def}For any $x,y$ in the same $W_1$ leaf, 
\begin{equation}
H^1_{c_1, x,y}=H^1_{c_3,x,y}=dh^{1}_{xy}
\end{equation}
Similarly, for any $x,y$ in the same $W_2$ leaf, 
\begin{equation}H^2_{c_2,x,y}=H^2_{c_4, x, y}=dh^2_{xy} 
\end{equation}
\end{lemma}
\begin{proof}Obviously $dh^i$ satisfies all the properties in Proposition \ref{prop: properties holonomies}, therefore by uniqueness of stable Holonomy of the cocycles $D\al(c_i)|_E$, we get the proof.
\end{proof}


\begin{lemma}\label{lemma: strong Osc space Hlmy inv}
For any $x\in M, y\in W_1(x), v\in E_x-\{0\}$, we have 
\begin{eqnarray}\label{eqn: same fwd behv}
\lim_{n\to \infty}\frac{\|D\al(nc_i)\cdot v\|}{\|D\al(nc_i)\cdot (dh^1_{xy}\cdot v)\|}=1, ~~i=1,3
\end{eqnarray}
Therefore the vectors $v$ and $dh^1_{xy}\cdot v$ have the same forward Lyapunov exponents (if they exist) under the action of  $D\al(c_i)|_E, i=1,3$. 
\end{lemma}
\begin{proof}Notice that $D\al(nc_i)\circ dh^1_{xy}=dh^{1}_{\al(nc_i)x, \al(nc_i)y}\circ D\al(nc_i)$, then by (3). in Proposition \ref{prop: properties holonomies} we get the proof.
\end{proof}

Suppose now $x,y$ are Oseledec generic points for the cocycle $D\al|_E$ with respect to an $\al-$invariant ergodic measure $\mu$, $E=E_{m_1\chi}\oplus \cdots \oplus E_{m_n\chi}$ $\mu-$almost everywhere, where $0<m_1<\cdots<m_n$. Then we have 
\begin{lemma}\label{lemma: emk uni Holo inv}
$E_{m_k\chi}(y)=dh^1_{xy}\cdot E_{m_k\chi}(x), k=1,\dots, n$. Similarly we have corresponding Holonomy invariance for Oseledec spaces $E_{m_k\chi}$ along $W_2$.
\end{lemma}
\begin{proof}Since $\chi(c_1)<0, \chi(c_3)>0$, we have that 
\begin{eqnarray*}
&&\{v\in E_y-\{0\}, \lim_{n\to \infty}\frac{1}{n}\ln\|D\al(n{c_1})\cdot v\|=m_k\chi(c_1)\} \\
&=& E_{m_k\chi}(y)\oplus \cdots \oplus E_{m_n\chi}(y)-E_{m_{k+1}\chi}(y)\oplus \cdots \oplus E_{m_n\chi}(y)
\end{eqnarray*}

\begin{eqnarray*}
&&\{v\in E_y-\{0\}, \lim_{n\to \infty}\frac{1}{n}\ln\|D\al(n{c_3})\cdot v\|=m_k\chi(c_3)\} \\
&=& E_{m_1\chi}(y)\oplus \cdots \oplus E_{m_k\chi}(y)-E_{m_{1}\chi}(y)\oplus \cdots \oplus E_{m_{k-1}\chi}(y)
\end{eqnarray*}
Therefore
$$E_{m_k\chi}(y)=\{0\}\cup \{v\in E_y, \lim_{n\to \infty}\frac{1}{n}\ln\|D\al(n{c_i})\cdot v\|=m_k\chi(c_i), i=1,3\}$$ 
Then by Lemma \ref{lemma: strong Osc space Hlmy inv} and equation \eqref{eqn: same fwd behv} we know $E_{m_k\chi}$ is invariant under $dh^1$. The proof for Holonomy-invariance along $W_2$ is similar.
\end{proof} 
 
\subsection{Ruelle's Theorem}\label{sec: Rue thm}
We recall the following version of a result by D. Ruelle, cf.  Theorem 6.3 in \cite{R79}, section 2 of \cite{Me}, or section 5.1 of \cite{KS17}.  Suppose $f$ is a $C^\infty$ Anosov diffeomorphism on a compact manifold $M$, and $\mu$ is an ergodic $f$-invariant measure. The Oseledec splitting $\oplus_i E_{\lambda_i}$ for cocycle $Df$ is well-defined $\mu-$almost everywhere, where 
$$\lambda_1 < \cdots < \lambda_s < 0< \lambda_{s+1}< \cdots $$
Moreover,
\begin{theorem}[Theorem 6.3 in \cite{R79}]\label{thm: r79}There exists a $\mu-$full measure $f-$invariant set $X\subset M$ such that for any $x\in X$, the flag of fast Oselededc subspaces is smooth along $W^s(x,f)$, i.e. there exists a smooth flag $$\cV_1\subset \cdots\subset\cV_s=E^s$$ defined on the whole stable leaf of $x\in X$ and for any $y\in W^s(x)\cap X, 1\leq k\leq s$, $$\cV_k(y)=\oplus_{i=1}^k E_{\lambda_i}(y)$$
\end{theorem}

\subsection{Non-uniformly fiber bunched cocycle and non-uniform Holonomy}\label{sec: non-uni coc Holo}
For (uniformly) fiber bunched H\"older continuous cocycle over a partially hyperbolic system we already defined the associated (un)stable Holonomy. In the following sections we will meet non-uniformly fiber bunched cocycles. We can define the associated non-uniform Holonomy map as well, cf. \cite{V08}. For simplicity, suppose $f$ is an $C^2$ Anosov diffeomorphism on a compact manifold $M$ and $\mu$ is an $f-$invariant probability measure on $M$. Let $F:E\to E$ be a H\"older continuous linear cocycle over $f$, where $E$ is a H\"older continuous vector bundle over $M$ (the same setting as in section \ref{sec: holo HOLO} and \cite{KS13}). As in Proposition \ref{prop: properties holonomies} for any two nearby points $x,y\in M$ we set $I_{xy}:E_x\to E_y$ be a linear identification which is H\"older close to the identity.

\begin{prop}\label{prop: nonunif holo def}If all the Lyapunov exponents of $F$ coincide, then for $\mu-$almost every $x\in M$, for any $y,z\in W^s_{loc}(x)$, for any $j\geq 0$, $$H^s_{f^j(y),f^j(z)}:=\lim_{n\to \infty}F^n(f^j(z))^{-1}\circ I_{f^{n+j}(y), f^{n+j}(z)}\circ F^n(f^j(y))=F^j(z)\cdot H^s_{yz}\cdot F^j(y)^{-1}$$
exists and is H\"older continuous along $W^s_{loc}$. Here the size of $W^s_{loc}(x)$ is \textbf{uniform} and only depends on $F$ and $f$. Similarly we could define $H^u$ along unstable manifold of $F$.
\end{prop}
\begin{proof}Proposition \ref{prop: nonunif holo def} is basically proved in \cite{V08}, section 2. The following are some slight differences between the settings here and that of \cite{V08}.

In section 2. of \cite{V08} the author consider the case that the bundle $E$ is trivial or smooth, but all the estimates there also work for H\"older continuous bundle. Moreover, since here our base dynamics is uniformly hyperbolic, we can choose the \textit{holonomy block} $\H(K,\tau)$ in \cite{V08} to be $M$. In addition, $W^s_{loc}(x)$ has uniform size for every $x$ and $f$ is uniform contracting on $W^s_{loc}$. Then the estimates in Proposition 2.5, Corollary 2.8 of \cite{V08} hold on $W^s_{loc}(x)$ for $\mu-$almost every $x$. As a result $H^s_{yz}$ is well-defined on $W^s_{loc}(x)$  for $\mu-$almost every $x$. 
\end{proof}
In \cite{V08}, Viana proved a non-uniform version of the \emph{invariance principle} \footnote{The word \emph{invariance principle} is used in \cite{AV10} for general smooth cocycles.} which was first obtained by Ledrappier \cite{L}, cf. \cite{V08}, Proposition 3.1. We denote the projectification of $H^s_{yz}$ by $h^s_{yz}$, and the projectification of $F:E\to E$ by $\PP(F):\PP(E)\to \PP(E)$. As in Proposition \ref{prop: nonunif holo def} we assume that all the Lyapunov exponents of $F$ are equal. Notice that the statement here is a weaker version of Proposition 3.1 in \cite{V08} since there $f$ is not necessary to be  Anosov.
\begin{prop}\label{prop: non unif inv prpl}
Let $m$ be any $\PP(F)-$invariant probability measure that projects down to $\mu$. Then the disintegration $m_z$ is invariant under the stable Holonomy $h^s$ $\mu-$almost everywhere. More precisely, there exists a full measure subset $E^s$ of $M$ such that for any $z_1,z_2\in M, z_2\in W^s_{loc}(z_1)$ if $z_1,z_2\in E^s$ then we have $$m_{z_2}=(h^s_{z_1,z_2})_\ast m_{z_1}$$ Similarly we have unstable Holonomy invariance for the disintegration $m_z$ of $m$.
\end{prop}
\subsection{Continuity of Oseledec filtration: local product structure and continuity}
\label{subsec: produ and cont}
In this section, we fix the measure $\mu$ which is the Bowen-Margulis measure of $\al(c_4)$, where $c_4$ is defined in Proposition \ref{prop: E12 splt}. Therefore $\mu$ has the local product structure with respect to the foliations $W_2$ and $W\oplus W_1$, i.e. locally $\mu \cong \mu^{W\oplus W_1}\times \mu^{W_2}$. In addition by uniqueness of Bowen-Margulis measure for transitive Anosov diffeomorphism we know that $\mu$ is $\al-$invariant.

Let $E=E_{m_1\chi}\oplus \cdots \oplus E_{m_n\chi}$, $0<m_1<\cdots<m_n$ be the $\mu-$almost everywhere defined Oseledec splitting of $E$ for the cocycle $D\al$, where $E$ is defined in section \ref{chap: spltt}. We claim that for any $k$, $1\leq k\leq n$, up to (redifining the values on) a $\mu-$neglible set, $E_{m_k\chi}\oplus \cdots \oplus E_{m_n\chi}$ is a H\"older continuous $D\al-$invariant bundle on $M$. 

The first step is the following corollary of Ruelle's theorem.
\begin{lemma}\label{lemma: non uni holo inv fast part}
Up to a modification on a $\mu-$neglible set, for $\mu-$almost every $x\in M$,  $E_{m_k\chi}\oplus \cdots \oplus E_{m_n\chi}$ is H\"older continuous on $W\oplus W_1(x)$. 
\end{lemma}
\begin{proof}
We consider an arbitrary Oseledec fast stable subspace  $\cV$ in $E\oplus E_1$ for cocycle $D\al(c_1)$. By Theorem \ref{thm: r79}, we know up to a modification on a $\mu-$neglible set, for $\mu-$almost every $x$, $\cV$ is smooth along $W\oplus W_1(x)$. Notice that $E_{m_k\chi}\oplus\cdots \oplus E_{m_n\chi}$ is the intersection of a fast subspace in $E\oplus E_1$ with the H\"older continuous subspace $E$ of $E\oplus E_1$. Hence up to a modification on a neglible set, $E_{m_k\chi}\oplus\cdots \oplus E_{m_n\chi}$ is H\"older continuous along $W\oplus W_1(x)$ for almost every $x$.
\end{proof}

Now we prove the main result of this subsection by creating a non-uniform \emph{redefining mechanism}, cf. Proposition 4.8 in \cite{AV10} for redefining argument for uniform fiber-bunched cocycles. Roughly speaking, consider a fiber bundle over a hyperblic space with a probability measure with local product structure, if a measurable section of the bundle is invariant under both of \textit{uniform unstable Holonomy} and \textit{non-uniform stable Holonomy}, then our redefining argument shows that it should coincide with a continuous section almost everywhere.

\begin{lemma}\label{lemma: product structure and continuity}The measurable subbundle $E_{m_k\chi}\oplus\cdots \oplus E_{m_n\chi}$ coincides with a H\"older continuous subbundle $\mu-$almost everywhere.
\end{lemma}

\begin{proof}For $\epsilon>0$, the foliation $W$ and any point $z\in M$, consider the leaf metric on $W$ induced by a smooth Riemannian metric on $M$. We denote by $W_\epsilon(z)$ the open $W-$ball centered $z$ with diameter $\epsilon$ (the leaf metric) on the $W-$leaf passes through $z$. Similarly we could define $(W_i)_\epsilon(z), (W\oplus W_i)_\epsilon(z)$. 

We firstly prove Lemma \ref{lemma: product structure and continuity} locally. i.e. for any $z\in M$ and  a ``rectangle" neighborhood $B(z)$ of $z$ with uniform size $\epsilon$, $B(z):=(W\oplus W_1)_{\epsilon}(z)\times (W_2)_{\epsilon}(z)$,  there is a continuous bundle  $\tE$ defined on $B(z)$ such that $\tE=E_{m_k\chi}\oplus\cdots \oplus E_{m_n\chi}$ almost everywhere. Here ``$\times$" is canonically defined by the local product structure of the foliations $W\oplus W_1$ and $W_2$.

 Here $\epsilon$ is chosen to be much smaller than the diameter of the neighborhood with product strucutre for $\al(c_4)$ and for any $z'\in B(z)$, $(W\oplus W_1)_{3\epsilon}(z')$ and $(W_2)_{3\epsilon}(z')$ are transverse to the boundary of $B(z)$. By Lemmas \ref{lemma: non uni holo inv fast part},  \ref{lemma: emk uni Holo inv}, there exists a $\mu-$full measure set $M_0$ of $M$ such that restricted on $M_0$, we have 
\begin{enumerate}
\item (Uniform Holonomy invariance along $W_2$) $E_{m_k\chi}\oplus\cdots \oplus E_{m_n\chi}$ is Holonomy invariant (under $dh^2$) along $W_2$ (by Lemma \ref{lemma: emk uni Holo inv}). 
\item (Non-uniform Holonomy invariance along $W\oplus W_1$) For any $x\in M_0$, the restriction on $(W\oplus W_1(x))\cap M_0$ of $E_{m_k\chi}\oplus\cdots \oplus E_{m_n\chi}|_{M_0}$ can be extended to be a H\"older continuous map along $W\oplus W_1(x)$ (by Lemma \ref{lemma: non uni holo inv fast part}).
\end{enumerate}

Since $\mu$ has local product structure, there exists $z_0
\in B(z)\cap M_0$ such that there is a full $\mu^{W\oplus W_1}-$ measure set $\cR(z_0)$ in $(W\oplus W_1(z_0))_{3\epsilon}\cap B(z)$ satisfying $\cR(z_0)\subset M_0$ and for any $z_1\in \cR(z_0)$, $(W_2)_{3\epsilon}(z_1)\cap M_0\cap B(z)$ has $\mu^{W_2}-$full measure in $(W_2)_{3\epsilon}(z_1)\cap B(z)$.

By our definition of $M_0$ (condition (2).), there is a H\"older continuous map from $\tE_0: (W\oplus W_1(z_0))_{3\epsilon}$ to the Grassmannian such that $\tE_0=E_{m_k\chi}\oplus\cdots \oplus E_{m_n\chi}$ on $\cR(z_0)$. Then the H\"older continuous bundle $\tE$ on $B(z)$ could be defined as the following:
for any $z'\in B(z)$, denote $(W_2)_{3\epsilon}(z')\cap (W\oplus W_1)_{3\epsilon}(z_0)$ by $z_1'$, then $$\tE(z'):=dh^2_{z_1'z'}\cdot \tE_0(z_1')$$

We claim that $\tE=E_{m_k\chi}\oplus\cdots \oplus E_{m_n\chi}$ on a $\mu-$full measure set in $B(z)$. In fact for any $z_1\in \cR(z_0)$ and $\mu^{W_2}-$almost every $z_2\in (W_2)_{3\epsilon}(z_1)\cap B(z)$, we have that $$E_{m_k\chi}\oplus\cdots \oplus E_{m_n\chi}(z_2)=dh^2_{z_1z_2}\cdot E_{m_k\chi}\oplus\cdots \oplus E_{m_n\chi}(z_1)=dh^2_{z_1z_2}\cdot\tE_0(z_1)=\tE(z_2)$$

Since $\mu$ has local product structure, the point $z_2$ we chose could run over a full $\mu-$measure set in  $B(z)$. Therefore $\tE=E_{m_k\chi}\oplus\cdots \oplus E_{m_n\chi}$ on a $\mu-$full measure set in $B(z)$. Notice that $\tE_0$ is H\"older continous along $(W\oplus W_1(z_0))_{3\epsilon} $ and $dh^2_{xy}$ H\"older continuously depends on $(x,y)$, therefore $\tE$ is H\"older continuous on $B(z)$. But $M$ could be covered by finite such rectangle neighborhoods $B(z)$, therefore $E_{m_k\chi}\oplus\cdots \oplus E_{m_n\chi}$ coincides with a H\"older continuous subbundle $\mu-$almost everywhere.
\end{proof}

In particular, the H\"older continuous extension of $E_{m_k\chi}\oplus \cdots \oplus E_{m_n\chi}$ defined in Lemma \ref{lemma: product structure and continuity} are $D\al-$invariant as well. In the rest of this section with a slight abuse of notation, we still denote by $E_{m_k\chi}\oplus \cdots \oplus E_{m_n\chi}$ the H\"older continuous extension of the associated Oseledec subbundle within $E$ .





\subsection{Zimmer's amenable reduction and its H\"older continuous extension}\label{subsec: cont zimm}
In this section we consider the H\"older continuous cocycle $D\al|_{E_{m_n\chi}}$, where $E_{m_n\chi}$ is a $D\al-$invariant H\"older continuous subbundle, defined by Lemma \ref{lemma: product structure and continuity}. By our assumption we know for any $c\in \ZZ^k$, the cocycle $D\al(c)$ has coinciding Lyapunov exponents within $E_{m_n\chi}$, with respect to the measure $\mu$ which is defined in the beginning of the subsection \ref{subsec: produ and cont}.

Our goal for this section is to prove the following key  lemma.

\begin{lemma}\label{lemma: cont red fin cov}There exists $N\in \ZZ^+$, a finite cover $\bar{M}$ of $M$ such that within $\bar{E}_{m_n\chi}$ (the lift of the subbundle $E_{m_n\chi}$ on $\bar{M}$), there exists a $D\bal|_{N\cdot \ZZ^k}-$invariant H\"older continuous flag 
$$\{0\}=\cE_{0}\subsetneqq \cE_{1}\subsetneqq\cdots \subsetneqq \cE_{l}=\bar E_{m_n\chi}$$
and a family of H\"older continuous metrics $\|\|_{i}$ within the bundles $\cE_{i}/\cE_{i+1}$ such that for any $v
\in \cE_{i}/\cE_{i+1}$, any $b\in \ZZ^k$,
\begin{equation}\label{eqn: main res in Emnchi}
\|D\bal(Nb)\cdot v\|_{i}=e^{m_n\chi(Nb)}\cdot \|v\|_{i}
\end{equation}
Here $\bal$ is a lift of $\al|_{N\cdot \ZZ^k}$ on $\bar M$.
\end{lemma}

\begin{proof}
Since $E_{m_n\chi}$ could be trivialized on a full $\mu-$measure set (cf. \cite{BP}, Proposition 1.2), we could measurably identify $E_{m_n}$ with $M\times \RR^{d}$, where $$d:=\dim E_{m_n\chi}$$ Then by Zimmer's amenable reduction theorem, cf. Corollary 1.8 of \cite{HK} or discussion in section 4.5 of  \cite{KS13},  there is a measurable coordinate change function $C: M\times \RR^{d}\to E_{m_n\chi}$ such that  $$A(x)=C(\al(c)\cdot x)^{-1}\cdot D\al(c)(x)\cdot C(x)\in G$$ for any $c\in \ZZ^k$ and $\mu-$a.e. $x$, where $G$ is an \emph{amenable} subgroup of $GL(d,\RR)$.

There are $2^{d-1}$ standard maximal amenable subgroups of $GL(d,\RR)$, each standard group corresponds to a composition of $d$, $$d=d_1+\cdots+ d_l$$
and consists of all block-triangular matrices of the form \begin{equation}\label{eqn: mtrx form zimmer}
\begin{pmatrix}
A_1&\ast&\cdots&\ast\\
&A_2&\cdots&\ast\\
&&\ddots&\vdots\\
&&&A_l
\end{pmatrix}
\end{equation}
where each diagonal block is a scalar multiple of orthogonal matrix. As in section 4.5 of \cite{KS13}, by Theorem 3.4 of \cite{Mo}, each amenable subgroup of $GL(d,\RR)$ has a finite index subgroup which is contained in a conjugate of one of the $2^{d-1}$ standard groups. Therefore (apply further measurable linear conjugacy if necessary) without loss of generality we may assume $G$ above has a finite index subgroup $G_0$ which is contained in one of the $2^{d-1}$ standard maximal amenable subgroups.\\

\textbf{Case 1}:
Firstly we consider the case that $G$ itself is contained in a standard subgroup of $GL(d,\RR)$. Then the subbundle $V_i$ spanned by the first $d_1+\cdots+d_i$ coordinate vectors in $\RR^d$ is $G-$invariant for $i=1,\dots, l$. As a corollary, there is a measurable $D\al-$invariant flag $$\{0\}=\cE_0\subsetneqq \cE_1\subsetneqq \cdots \subsetneqq \cE_l=E_{m_n\chi}$$ and a measurable coordinate change within $\cE_i/\cE_{i-1}$ such that the measurable cocycle $D\al$ on the bundle $\cE_i/\cE_{i-1}$ under this coordinate change is a scalar multiple of an orthogonal matrix $\mu-$almost everywhere. We call the measurable flag defined above and associated measurable conformal strucutures  (invariant angle function on $\cE_i/\cE_{i-1}$) the \emph{measurable Zimmer amenable reduction} for cocycle $D\al|_{E_{m_n\chi}}$, cf. \cite{KS13} for more details. The first step is the following lemma.

\color{black}

 \begin{lemma}\label{lemma: hold cont cEi}For any $i$, $\cE_i$ coincides with a H\"older continuous subbundle $\mu-$almost everywhere.
\end{lemma}
\begin{proof}
The strategy to prove H\"older continuity of $\cE_i$ is similar to the proof of Lemma \ref{lemma: product structure and continuity}, i.e. we will use the redefining argument again. The main difference is that, rather than applying Ruelle's result to produce non-uniform unstable Holonomy, here we use the non-uniform version of invariance principle, i.e. Proposition \ref{prop: non unif inv prpl} to construct non-uniform unstable Holonomy we need. 

We denote by $\dim \cE_i=d_i$ and consider the cocycle $D\al(c_4)|_{E}$. By estimate \eqref{eqn:FiberBu}, $D\al(c_4)|_{E}$ is uniformly fiber-bunched, and by Lemma \ref{lemma: emk uni Holo inv}, the stable Holonomy $dh^2$ along $W_2$ for cocycle $D\al(c_4)|_{E}$ preserves $E_{m_n\chi}$. 

By H\"older continuity of $E_{m_n\chi}$, the H\"older continuous cocycle $\wedge^{d_i}D\al(c_4)|_{E_{m_n\chi}}$ is well-defined. Moreover if $c_4$ is chosen to be sufficiently close to $L$, then $\wedge^{d_i}D\al(c_4)|_{E_{m_n\chi}}$ is also uniformly fiber-bunched on $W_2$ (by estimate \ref{eqn:FiberBu}) and has uniformly stable Holonomy along $W_2$. By uniqueness of Holonomy for fiber-bunched H\"older continuous cocycle (Proposition \ref{prop: properties holonomies}) we know the stable Holonomy for $\wedge^{d_i}D\al(c_4)|_{E_{m_n\chi}}$ is exactly $\wedge^{d_i}(dh_{xy}^2|_{E_{m_n\chi}})$ along $W_2$, i.e. for any $x,y$ on the same local $W_2-$leaf, we have (by b. of Proposition \ref{prop: properties holonomies})
\begin{equation}\label{eqn: unif Holo of wedge di c4 mnchi}
\wedge^{d_i}(dh_{xy}^2|_{E_{m_n\chi}})=\lim_{n\to \infty}(\wedge^{d_i}D\al(nc_4)|_{E_{m_n\chi}})(y)^{−1}\circ I_{\al(nc_4)\cdot x, \al(nc_4)\cdot y  }\circ \wedge^{d_i}D\al(nc_4)|_{E_{m_n\chi}}(x)
\end{equation}where as in definition of Holonomy of fiber-bunched cocycle, $$I_{\al(nc_4)\cdot x, \al(nc_4)\cdot y  }:\wedge^{d_i}(E_{m_n\chi})(\al(nc_4)\cdot x)\to \wedge^{d_i}(E_{m_n\chi})(\al(nc_4)\cdot y)$$ is a linear identification H\"older close to the identity.
Moreover $\wedge^{d_i}(dh_{xy}^2|_{E_{m_n\chi}})$ is uniformly H\"older continuous in $(x,y)$.

On the other hand, since the cocycle $\wedge^{d_i}D\al(c_4)|_{E_{m_n\chi}}$ has coincide Lyapuonv exponents within $\wedge^{d_i}(E_{m_n\chi})$ with respect to the Bowen-Margulis measure $\mu$ for $\al(c_4)$, by Proposition \ref{prop: nonunif holo def} we know it admits non-uniformly stable and unstable Holonomies. The non-uniform unstable Holonomy of $\wedge^{d_i}D\al(c_4)|_{E_{m_n\chi}}$ is H\"older continuous along $(W\oplus W_1)_{loc}(x)$ for $\mu-$almost every $x$ by Proposition \ref{prop: nonunif holo def}. But by our discussion in last paragraph we know $\wedge^{d_i}D\al(c_4)|_{E_{m_n\chi}}$ actually admits \textbf{uniform} stable Holonomy along $W_2$ since $\wedge^{d_i}D\al(c_4)|_{E_{m_n\chi}}$ is uniform fiber-bunched as mentioned in last parapraph. Therefore the \textbf{non-uniform} stable holonomy (induced by Proposition \ref{prop: nonunif holo def} and equality of Lyapunov exponents) of $\wedge^{d_i}D\al(c_4)|_{E_{m_n\chi}}$ along $W_2$ coincides with the \textbf{uniform} stable Holonomy defined in \eqref{eqn: unif Holo of wedge di c4 mnchi}. In the rest of the proof of Lemma \ref{lemma: hold cont cEi} we do not distinguish the uniform and non-uniform stable Holonomy. In particular, the stable Holonomy is uniformly H\"older continuous in $(x,y)$. 

By Proposition \ref{prop: non unif inv prpl} any invariant measure $m$ on $\PP(\wedge^{d_i}(E_{m_n\chi}))$ is $\mu-$almost surely invariant under stable and unstable Holonomy. Notice that each $\cE_i$ can be identified with an invariant measure $m$ on the bundle $\PP(\wedge^{d_i}(E_{m_n\chi}))$ and the cocycle on $\wedge^{d_i}(E_{m_n\chi})$ induced by $D\al$ has coinciding Lyapunov exponents (with respect to the measure $\mu$). Therefore $\cE_i$ is $\mu-$almost surely invariant under stable and unstable Holonomy. 

Now we use the redefining argument as in the proof of Lemma \ref{lemma: product structure and continuity}. Notice that as in the proof of \ref{lemma: product structure and continuity}, the stable Holonomy is uniformly H\"older continuous and well-defined everywhere. And as in Lemma \ref{lemma: product structure and continuity} we only need to prove Lemma \ref{lemma: hold cont cEi} locally, i.e. for any $B(z)$ as in the proof of \ref{lemma: product structure and continuity}, there is a H\"older continuous bundle $\tcE_i$ defind on $B(z)$ such that $\tcE_i=\cE_i$ $\mu-$almost everywhere. 

Recall that $\mu$ is the Bowen-Margulis measure of $\al(c_4)$. By Proposition \ref{prop: nonunif holo def} and Proposition \ref{prop: non unif inv prpl} we know that there exists a $\mu-$full measure set $M_0$ such that restricted on $M_0$, $\cE_i$ is Holonomy invariant along $W_2$ (under $\wedge^{d_i}dh^2_{xy}|_{E_{m_n\chi}}$) 
 on the local $W_2$ leaves. And for any $x\in M_0$, the restriction on $(W\oplus W_1(x))\cap M_0$ of $\cE_i$ can be extended to be a H\"older continuous map along $W\oplus W_1(x)$ (Apply Proposition \ref{prop: nonunif holo def} and Proposition \ref{prop: non unif inv prpl} to cocycle $\wedge^{d_i}D\al(c_4)|_{E_{m_n\chi}}$ and the invariant measure $m$ on $\PP(\wedge^{d_i}(E_{m_n\chi}))$ corresponding to the subspace $\cE_i$).

Then we could find a point $z_0$ and a full $\mu^{W\oplus W_1}-$measure set $\cR(z_0)$ satisfies the same conditions as in the proof of Lemma \ref{lemma: product structure and continuity}. The rest proof of existence for H\"older extension of $\cE_i$ is the same as that of Lemma \ref{lemma: product structure and continuity}, we only need to replace $dh^2$ there by $\wedge^{d_i}dh^2|_{E_{m_n\chi}}$ here.
\end{proof} 
 
By Lemma \ref{lemma: hold cont cEi} we know, up to a modification on a $\mu-$neglible set, $\cE_i$ has a H\"older continuous extension on $M$, which we still denote by $\cE_i$. Then we get a $D\al-$invariant H\"older continuous flag $$\{0\}=\cE_0\subsetneqq \cdots \subsetneqq \cE_l=E_{m_n\chi}$$ therefore $\cE_{i}/\cE_{i-1}$ are also H\"older continuous vector bundles over $M$. 

A priori the measurable conformal structure $\mC_i$ (i.e. an invariant angle function on $\cE_i/\cE_{i-1}$) on each $\cE_{i}/\cE_{i-1}$ induced by measurable Zimmer amenable reduction is not necessary to be  continuous. A remarkable fact is that each $\mC'_i$ can be identified with a $\PP(\wedge^{d_i-d_{i-1}}D\al(c_4)|_{\cE_{i}/\cE_{i-1}})-$invariant measure on the bundle $\PP(\wedge^{d_i-d_{i-1}}{\cE_{i}/\cE_{i-1}})$. If $c_4$ is chosen to be close enough to $L$, as previous discussion, the cocycle $\wedge^{d_i-d_{i-1}}D\al(c_4)|_{\cE_{i}/\cE_{i-1}}$ is uniformly fiber-bunched over $W_2$. Then it admits a \textbf{uniform} stable Holonomy on any local $W_2-$leaf: $\wedge^{d_i-d_{i-1}}(d\bar{h}_{xy}^2|_{\cE_{i}/\cE_{i-1}})$ which depends uniformly H\"older continuously  on $(x,y)$ on the same $W_2$ local leaf. Again by Proposition \ref{prop: nonunif holo def}, we know $\wedge^{d_i-d_{i-1}}D\al(c_4)|_{\cE_{i}/\cE_{i-1}}$ admits \textbf{non-uniform} unstable holonomy $\mu-$almost everywhere. In particular the unstable Holonomy is H\"older continuous along $(W\oplus W^1)_{loc}(x)$ for $\mu-$almost every $x$. Then by Proposition  \ref{prop: non unif inv prpl} we know $\mC_i$ is $\mu-$almost surely invariant 
under stable and unstable Holonomies. Therefore by mimicking the proofs of Lemmas \ref{lemma: product structure and continuity}, \ref{lemma: hold cont cEi},   we get that $\mC_i$ has an invariant H\"older continuous extensions on $M$, which we still denote by $\mC_i$. 

\textbf{Case 2}: Now we consider the case that that only a finite index subgroup $G_0$ of $G$ is contained in a standard subgroup $G_1$. Actually our argument for  Case 2 is basically similar to that of \cite{KS13}. The basic idea to construct corresponding continuous flags and continuous invariant conformal structures on the factor bundles for each flag. In general this requires to pass to a finite cover and take restriction to a finite index subgroup of $\ZZ^k$. For completeness we give a proof in the appendix. 
\end{proof}





\subsection{Closing lemma and polynomial deviation}\label{sec: cls lem}By discussions in previous sections (passing to a finite cover and taking restriction to a finite index subgroup of $\ZZ^k$ if necessary), we have a H\"older continuous $D\al$ invariant flag $\{\cE_i\}$ within $E_{m_n\chi}$, and a family of H\"older continuous invariant conformal structures $\{\mC_i\}$ on the bundles $\cE_i/\cE_{i-1}$.  

In this section we prove that $D\al$ has polynomial deviation within bundle $E_{m_n\chi}$. The first step is to  construct a H\"older continuous metric $\|\cdot\|_i$ within the bundle $\cE_i/\cE_{i-1}$ such that for any $v\in \cE_i/\cE_{i-1}$, for any $b\in \ZZ^k$, 
\begin{equation}\label{eqn: unif expand Ei-Ei-1}
\|D\al(b)\cdot v\|_i=e^{m_n\chi(b)}\cdot \|v\|_i
\end{equation}
The proof is  a closing-lemma argument which is essentially contained in section 3.2 of \cite{KS07} and in \cite{KalSpa}. Or cf. Lemma 7.4 \cite{DX17}, for completeness we give an outline here.

Consider the suspension $\tal$ on $\tM$ of $\al$, since all nontrivial elements of $\al$ are Anosov and $\al$ is TNS, $\al$ satisfying all assumptions of Proposition 2.3. in \cite{KS07}. We still denote the natural extension of $\cE_i/\cE_{i-1}$ on $\tM$ by $\cE_i/\cE_{i-1}$. 
As in \cite{KS07}, we fix a smooth background metric ${g}_0$ on $\cE_i/\cE_{i-1}$. Then ${g}_0$ induced a volume form ${\nu}_0$ on $\cE_i/\cE_{i-1}$, we defined a H\"older continuous function $q$ as the following: for any $x\in \tilde{M}, b\in \RR^k$,
\begin{equation}\label{eqn: def q clos}
q(x,b):=\mathrm{Jac}_{{\nu}_0}(D\tilde{\al}(b)_x|_{\cE_i/\cE_{i-1}})
\end{equation}

Since $\al$ is totally non-symplectic, we can choose a \textit{generic singular} element (i.e. it contained in exactly one Lyapunov hyperplane) (cf. \cite{KS07, KalSpa}) $a_0\in \RR^k$ in the Lyapunov hyperplane $L$ such that  $\tilde{\al}(ta_0), t\in \RR$ acts transitively on $\tilde{M}$. Our goal is to prove the following lemma:
\begin{lemma}\label{lemma: coh eqn closing}There is a H\"older continuous positive function ${\phi}$ on $\tilde{M}$ such that for any $x\in \tilde{M}$
\begin{equation}\label{equation: coho eqn gene sing a}
{\phi}(x)\cdot{\phi}(\tilde{\al}(a_0)\cdot x)^{-1}=q(x,a_0)
\end{equation}
\end{lemma}
We give a proof of Lemma \ref{lemma: coh eqn closing} in the appendix. Now we claim that for any $x\in \tilde{M}, b\in \RR^k$,
\begin{equation}\label{equation: coho eqn all b}
{\phi}(x)\cdot{\phi}(\tilde{\al}(b)\cdot x)^{-1}=e^{-d\chi(b)}q(x,b)
\end{equation}
where $\chi$ is the Lypunov functional corresponding to $L$ and $d$ is the dimension of $\cE_i/\cE_{i-1}$. In fact, we can define a H\"older continuous volume form on $\cE_i/\cE_{i-1}$ by ${\nu}:={\phi}\cdot {\nu}_0$. By \eqref{equation: coho eqn gene sing a}, ${\nu}$ is invariant under the action of $\tilde{\al}(ta_0), t\in \RR$. By commutativity, $\tilde{\al}(b)_\ast {\nu}=\psi\cdot {\nu}$ is also an $\tilde{\al}(ta_0)-$invariant volume form on $\cE_i/\cE_{i-1}$. Since $\al(ta)$ acts transitively on $\tilde{M}$, $\psi$ is a constant function. Notice that $b$ acts on $\cE_i/\cE_{i-1}$ with Lyapunov exponents $\chi(b)$ (since we already proved that $\al$'s action is uniformly quasiconformal on $\cE_i/\cE_{i-1}$), by constancy of $\psi$ we get \eqref{equation: coho eqn all b}. Therefore, combining the volume form $\nu$ with the conformal structure we got in section \ref{subsec: cont zimm} on $\cE_i/\cE_{i-1}$, we easily get the H\"older continuous metric we need.



As a corollary, by a H\"older continuous change of coordinates, for any $b\in \ZZ^k$ the cocycle $D\al(b)|_{E_{m_n\chi}}$ has the form
$$\begin{pmatrix}
e^{m_n\chi(b)}\cdot O_1&\ast&\cdots&\ast\\
&e^{m_n\chi(b)}\cdot O_2&\cdots&\ast\\
&&\ddots&\vdots\\
&&&e^{m_n\chi(b)}\cdot O_l
\end{pmatrix}
$$where $O_i$ are orthogonal matrices and all the "$\ast$" terms are uniformly bounded. Therefore by following lemma in linear algebra we can prove the \textbf{polynomial deviation} for $D\al$ within bundle $E_{m_n\chi}$.

\begin{lemma}Suppose $\{A_k=B_k+C_k, k\geq 0\}$ is a sequences of block upper triangular square matrices (with the same form for each $k$), $B_k$ and $C_k$ are the diagonal and off-diagonal part of $A_k$ respectively. If $\|C_k\|$ are uniformly bounded and $B_k$ are orthogonal matrices, then there exist $C>0$ and $n\in \ZZ^+$ such that for any $m\in \ZZ^+$, 
$$\|A_m\cdots A_2A_1 \|\leq Cm^n, \|(A_m\cdots A_2A_1 )^{-1}\|\leq Cm^n$$
\end{lemma}
\begin{proof}Consider the full expansion of $$A_m\cdots A_2A_1=(B_m+C_m)\cdots (B_1+C_1)$$ There are $2^m$ terms, but by the form of $A_k$ we know there exists $n\in \ZZ^+$ (in fact $n$ could be chosen as the size of $A_k$) such that  any term in the sum which contains $n+1$ $C_m$ should be $0$. Therefore $$\|A_m\cdots A_2A_1 \|\leq \binom{m}{n}\cdot \sup_{k\geq 1}\|C_k\|^n=O(m^n)$$
The same argument also works for $\|(A_m\cdots A_2A_1 )^{-1}\|$ since $A_k^{-1}$ has the same form as $A_k$ and the off-diagonal parts of $A_k^{-1}, k\geq 1$ are uniformly bounded as well.
\end{proof}

In summary, we prove the H\"older continuity of $E_{m_n\chi}$ and get the \emph{standard form} within $E_{m_n\chi}$. In addition, the cocycle $D\al|_{E_{m_n\chi}}$ has polynomial deviation.

\subsection{Cone criterion, H\"older continuity of Oseledec splitting and proof of Theorem \ref{thm: osl spltt cont}}Now we consider the bundles $$(E_{m_k\chi}\oplus \cdots \oplus E_{m_n\chi})/(E_{m_{k+1}\chi}\oplus \cdots \oplus E_{m_n\chi})$$ for $1\leq k\leq n-1$. By Lemma \ref{lemma: product structure and continuity} we know the restriction of $D\al$ on $(E_{m_k\chi}\oplus \cdots \oplus E_{m_n\chi})/(E_{m_{k+1}\chi}\oplus \cdots \oplus E_{m_n\chi})$ is a H\"older continuous cocycle with coinciding Lyapunov exponents. Therefore, mimicking the proof of Lemma \ref{lemma: hold cont cEi} and discussion in section \ref{sec: cls lem} we could get similar standard form in $(E_{m_k\chi}\oplus \cdots \oplus E_{m_n\chi})/(E_{m_{k+1}\chi}\oplus \cdots \oplus E_{m_n\chi})$ (passing to a finite cover and taking a restriction to a finite index subgroup of $\ZZ^k$ if necessary). In particular the cocycle $D\al$ on $(E_{m_k\chi}\oplus \cdots \oplus E_{m_n\chi})/(E_{m_{k+1}\chi}\oplus \cdots \oplus E_{m_n\chi})$ has polynomial deviation and the associated Lyapunov exponents are $m_k\chi(\cdot)$.

Choose $b\in \ZZ$ such that $\chi(b)<0$, and consider the cocycle $D\al(b)$ on the bundle $E_{m_{n-1}\chi}\oplus E_{m_{n}\chi}$, notice that the speed of the action of $D\al(b)$ on $$E_{m_{n-1}\chi}\oplus E_{m_{n}\chi}/E_{m_{n}\chi}$$ is uniformly exponentially faster (thanks to the standard form) than that on $E_{m_n\chi}$, then by cone criterion in section 2. of \cite{CP} and the proof of Lemma 5.3 of \cite{DX}, we will prove that there is a dominated splitting within $E_{m_{n-1}\chi}\oplus E_{m_{n}\chi}$ and one subspace is $E_{m_n\chi}$.

The first step is the following lemma which implies that there is a invariant cone field within $E_{m_{n-1}\chi}\oplus E_{m_n\chi}$. 
\begin{lemma}\label{lemma: app cone}
Suppose $f:X\to X$ is a homeomorphism on a compact metric space, $E$ is a (normed) vector bundle over $X$ and $F$ is a continuous (invertible) linear cocycle $E\to E$ over $f$. We assume that there is a continuous splitting of $E=E_1\oplus E_2$ where $E_1$ is $F-$invariant. Denote by $p_i,i=1,2$ the canonical linear projection from $E$ onto $E_i$. Suppose that $$\inf_{x\in M, v_2\in E_2, v_1\in E_1, \|v_2\|=\|v_1\|=1}\frac{\|p_2( F\cdot v_2)\|}{\|F_1\cdot v_1\|}>\lambda>1$$
Then there exist $ \gamma>0, \epsilon\in (0,1)$  such that for the cone field 
$$\C_\gamma:=\{(v,u)\in E_1\oplus E_2, \|v\| \leq \gamma\cdot \|u\| \}$$
we have $$F_x\cdot \C_\gamma(x)\subset \C_{(1-\epsilon)\gamma}(f(x)).$$  
\end{lemma}
\begin{proof}
Suppose that $(u_x,v_x)\in E_2(x)\oplus E_1(x)$, then we have the following matrix form for $F$, $$(u_{f(x)},v_{f(x)})=F_x\cdot (u_x,v_x)^t=\begin{pmatrix}
A(x)&\\C(x) &D(x)
\end{pmatrix}\cdot (u_x,v_x)^t$$
where $$A(x)= p_2\circ F_x|_{E_2(x)}), C(x)= p_1\circ F_x|_{E_2(x)}, D(x):=F_x|_{E_1(x)}.$$ 
Moreover by the assumption of Lemma \ref{lemma: app cone} we know $$\lambda\cdot \|D(x)\|<\|A(x)^{-1}\|^{-1}=\min_{u\in E_2(x), \|u\|=1}\|A(x)\cdot u\| $$
Then for $\gamma$ sufficiently large $\epsilon$ sufficiently small, if $\|v_x\|\leq \gamma\cdot \|u_x\| $ we have that 
\begin{eqnarray*}
\|v_{f(x)}\|&\leq& \|C(x)\|\cdot \|u_x\|+\|D(x)\|\cdot \|u_x\|\cdot \gamma\\
&\leq& \|C(x)\|\cdot \|A(x)\|^{-1}\cdot \|u_{f(x)}\|+\gamma \cdot \|D(x)\|\cdot \|A(x)\|^{-1}\cdot \|u_{f(x)}\|\\
&\leq& (\lambda^{-1}\gamma+\|C(x)\|\cdot \|A(x)^{-1}\|)\cdot \|u_{f(x)}\|\\
&\leq & (1-\epsilon)\gamma\cdot \|u_{f(x)}\| \text{ for $\gamma$ sufficiently large and $\epsilon $ sufficiently small}\end{eqnarray*}\end{proof}
Coming back to the discussion for the cocycle $D\al(b)$ on $E_{m_{n-1}\chi}\oplus E_{m_n\chi}$, where $\chi(b)<0$. Consider the inner product on $E_{m_{n-1}\chi}\oplus E_{m_n\chi}$ by the metric on $M$ and denote the orthogonal complement $(E_{m_n\chi})^{\perp}$ of $E_{m_n\chi}$ in $E_{m_{n-1}\chi}\oplus E_{m_n\chi}$. Notice that since $m_{n-1}<m_n$, the cocycle induced by $D\al(b)$ on $E_{m_{n-1}\chi}\oplus E_{m_n\chi}/E_{m_n\chi}$ has growth uniformly exponentially faster than that on $E_{m_n\chi}$, therefore for  $l$ large enough, we could apply Lemma \ref{lemma: app cone} with $$f=\al(lb), E=E_{m_{n-1}\chi}\oplus E_{m_n\chi}, F=D\al(lb)|_{E_{m_{n-1}\chi}\oplus E_{m_n\chi}}, E_1= E_{m_n\chi}, E_2= (E_{m_n\chi})^{\perp}$$
Then by classical cone criterion in the theory of hyperbolic dynamics (cf. \cite{CP}) we know there is a unique $D\al(b)-$invariant distribution within $E_{m_{n-1}\chi}\oplus E_{m_n\chi}$ has the same dimension as $E_{m_{n-1}\chi}$ and uniformly transverse to $E_{m_n\chi}$, and the associated splitting is dominated. Therefore the new invariant bundle we got should be almost surely coinciding with the measurable invariant bundle $E_{m_{n-1}\chi}$, which is the result we need.


By the theory of dominated splitting, $E_{m_{n-1}\chi}$ is H\"older continuous, cf. \cite{CP}. Moreover $D\al$ has coincide Lyapunov exponents within $E_{m_{n-1}\chi}$, therefore by repeating the discussion for $E_{m_n\chi}$ we can prove that $D\al$ has a standard form within $E_{m_{n-1}\chi}$ (passing to a finite cover and taking a restriction to a finite index subgroup of $\ZZ^k$ if necessary). Repeating the argument above we could prove the H\"older continuity of each $E_{m_k\chi}$ and within each $E_{m_k\chi}$, $D\al$ preserves a standard form (passing to a finite cover and taking a restriction to a finite index subgroup of $\ZZ^k$ if necessary), which completes the proof of Theorem \ref{thm: osl spltt cont}.

\section{The pairwise jointly integrable case and the Lyapunov pinching case }\label{sec: JointlyIntegrable}

\subsection{ Result of Brin-Manning and the proof of Theorem \ref{main Lyapunov pinching}}Recall that in \cite{BM}, Brin and Manning proved that an Anosov diffeomorphism $f$ is topologically conjugate to an affine automorphism on an infranilmanifold if $f$ satisfying certain ``pinching" condition for the \emph{Mather spectrum} (for precise definition and more details cf. \cite{BM}, \cite{HPS}, etc.), i.e. if the Mather spectrum of $f$ is contained in the union of two annuli with radii $0<r_1<r_2<1$ and $1<R_2<R_1<\infty$ where 
\begin{equation}\label{eqn: pinched M spec}
1+\frac{\ln R_2}{\ln R_1}>\frac{\ln r_1}{\ln r_2} ~~\text{or}~~1+\frac{\ln r_2}{\ln r_1}>\frac{\ln R_1}{\ln R_2}
\end{equation}
By Theorem \ref{thm: osl spltt cont} and definition of Lyapunov pinching condition we know, if $\al$ is Lyapunov pinching then there is an element $a$ such that the Mather spectrum of $\al(a)$ satisfies \eqref{eqn: pinched M spec} \footnote{By polynomial deviation property in Theorem \ref{thm: osl spltt cont}, there exist a finite cover $\bar M$ of $M$, $N\in \ZZ^+$ such that the Mather spectrum of the lift $\bal(Na)$ on $\bar M$ satisfies \eqref{eqn: pinched M spec}. Notice that the covering map is a local isometry so the Mather spectrum of $\al(Na)$ also satisfies \eqref{eqn: pinched M spec}.   }. Therefore by results of Brin and Manning we know $\al(a)$ is topologically conjugate to an affine action on an infranilmanifold. Therefore the ambient action of $\al$ is topologically conjugate to a $\ZZ^k-$action $\rho$ by affine actions on an infranilmanifold as well. Notice that $\al$ is TNS, then by Lemma 4.1 of \cite{DX} we know $\rho$ is TNS as well and $\rho$ can not has a rank-one factor. Therefore by Corollary 1.2 in  \cite{HW}, $M$ should be standard infranilmanifold. As a result, by global rigidity of higher rank smooth Anosov action on infranilmanifold \cite{FKS, HW} we know $\al$ is smoothly conjugate to the algebraic action $\rho$.

\subsection{Polynomial global product structure and the proof of Theorem \ref{main toral case}\color{black}}
Recall that two foliations $\cF_1,\cF_2$ have \textit{global product structure} if every leaf $\cF_1(x)$ intersects every leaf $\cF_2(y)$ in a unique point $[x,y]$. Moreover, for foliations $\cF_1$ and $\cF_2$ have global product structure, we say they have \textit{polynomial global product structure} if there is a polynomial
$p$ such that $$d_{\cF_1}(x, [x, y])+d_{\cF_2}(y, [x, y]) < p(d(x, y))$$ for all $x$ and $y$, where $d_{\cF_i}$ are the Riemannian distances along leaves and $d$ is the Riemannian distance on the ambient manifold. 

An Anosov diffeomorphism is called \textit{has (polynomial) global product structure} if the stable and unstable foliations have (polynomial) global product structure on the universal cover.

In \cite{Ha}, Hammerlindl proved that an Anosov diffeomorphism is topologically conjugate to an infranilmanifold automorphism if it has the polynomial global product structure. The main ingredients appear
in the work of Brin and Manning \cite{B78}, \cite{BM}. 

\begin{proof}[Proof of Theorem \ref{main toral case}] we will prove in this section that any Anosov element of $\al$ has global polynomial  product structure. Therefore by \cite{Ha} the ambient action of \color{black} $\al$ is topologically conjugate to an $\ZZ^k-$action $\rho$ by affine actions on infranilmanifold. Then  as the proof of Theorem \ref{main Lyapunov pinching}, by \cite{HW}, $\al$ is smoothly conjugate to $\rho$.

We fix an Anosov element $f$ of $\al$ on $M$, the first step is to prove that $f$ has a global product structure. 

Denote $[0,1]$ by $I$. Brin's proof is basically as the following, in \cite{B77} he called the \textit{corresponding mapping} (or "holonomy" in the language of our paper) of stable and unstable manifolds $W^s=W^s(f), W^u=W^u(f)$ \textit{can be infinitely extended} if for every three points $x\in M, y\in W^s(x), z\in  W^u(x)$, there exists a continuous mapping $g$ from unit square $I^2$ into $M$ such that 
\begin{enumerate}
\item $g(0,0)=x, g(0,1)=y, g(1,0)=z$.
\item $g(t,\cdot)$ is a continuous curve on a stable leaf for every fixed $t\in I$.
\item $g(\cdot,t)$ is a continuous curve on a unstable leaf for every fixed $t\in I$.
\end{enumerate}
Brin proved that under pinching condition \eqref{eqn: pinched M spec}, the corresponding mapping of stable and unstable foliations can be infinitely extended. Moreover Brin showed that for any Anosov diffeomorphism such that the corresponding mapping of stable and unstable foliations can be infinitely extended has a global product structure. 

We claim that the corresponding mapping  of stable and unstable foliations of $f$ can be infinitely extended. The proof of is not hard. The following lemma is a corollary of topological joint integrability of coarse Lyapunov foliations.
\begin{lemma}\label{lemma: prod stru on wij}
For any pair of coarse Lyapunov foliations $W_1, W_2$, if we denote by $W$ the integrable foliation of $W_1$ and $W_2$, then the leaves of $W_1$, $W_2$ have global product structure on each leaf of $W$. Moreover $W$ has uniformly $C^\infty-$leaves which tangent to $TW_1\oplus TW_2$.
\end{lemma}
\begin{proof}cf. \cite{KS06}. 
\end{proof}

\begin{coro}\label{coro: smooth Wi}Under the assumptions in Theorem \ref{main toral case}, all the coarse Lyapunov distributions are uniformly $C^\infty$.
\end{coro}
\begin{proof}By Journ\'e lemma \cite{Journe}, we only need prove that for any pair of coarse Lyapunov foliations $W_1$ and $W_2$, $TW_1$ is uniformly smooth along $W_2$. The proof is similar to that of Proposition \ref{prop: E12 splt}. By Lemma \ref{lemma:sub exp grwth on plane} and the proof of Proposition \ref{prop: E12 splt}, we can find an Anosov element of $\al$ which contracts $W_1$ than it does $W_2$. Therefore $W_1$ is uniformly smooth within $W_1\oplus W_2$ (the integrable foliation of $W_1$ and $W_2$), which implies $TW_1$ is uniformly smooth along $W_2$.
\end{proof}
Now we come back to the discussion of $f$. The following lemma is a corollary of smoothness of coarse Lyapunov distributions and topological joint-integrability, 
\begin{lemma}\label{lemma: prod Eisu}Any direct sum of coarse Lyapunov distributions are integrable.
\end{lemma}
\begin{proof}By smoothness of coarse Lyapunov distributions and joint-integrability we know for any smooth vector fields $X_1, X_2$ tangent to some coarse Lyapunov distributions $E_1, E_2$ respectively, $[X_1,X_2]\subset E_1\oplus E_2$. Therefore for any direct sum of coarse Lyapunov distributions $E_1\oplus \cdot \oplus E_k$, we pick arbitrary pair of smooth vector fields $Y_1, Y_2$ which contained in $\oplus_{i=1}^k E_i$, we get $$[Y_1, Y_2]=\sum_{i,j}[Y_{1i}, Y_{2j}]\subset E_1\oplus \cdots \oplus E_k $$here $Y_{1i}$ ($Y_{2j}$) is the linear projection of $Y_1$ to $E_i$ ($E_j$ resp.). Therefore by Frobenius theorem we get the proof. 
\end{proof}

Suppose that $$E^s(f)=\oplus E^s_i, E^u(f)=\oplus E^u_i$$
where all of $E^s_i, E^u_i$ here are coarse Lyapunov distributions, and the associated Lyapunov foliations are denoted by $W^s_i, W^u_i$ respectively. By Lemma \ref{lemma: prod Eisu} we know locally for any $x,y$ on $M$ we could connect them by ``accessible" path, i.e. a  piecewise $C^1$ path $\gamma$ from $x$ to $y$ such that each $C^1-$part of $\gamma$ is contained in some coarse Lyapunov foliation. But since $f$ uniformly contracts $W^s(f)$, therefore the local structure of foliations on $W^s(f)$ is the same as global structure. As a corollary we have for any $x,y$ in the same $W^s(f)-$leaf there is an accessible path from $x$ to $y$ which is contained in $W^s(f,x)$.




For any coarse Lyapunov foliation $W$ in $W^u(f)$, we consider the foliations $W, W\oplus W^s(f)$. Then we have
\begin{lemma}\label{lemma: toral case induction} For every three points $x\in M, y\in W^s(x), z\in W$ and continuous paths $$\gamma_1: I\to W^s(x), \gamma_1(0)=x, \gamma_1(1)=y; \gamma_2:I\to W(x), \gamma_2(0)=x, \gamma_2(1)=z $$ such that $\gamma_1$ is accessible, there exists a continuous mapping $g$ from unit square $I^2$ into $M$ such that \begin{enumerate}
\item $g(0,0)=x, g(0,1)=y, g(1,0)=z$.
\item $g(t,\cdot)$ is an accessible path on a $W^s(f)$ leaf for every fixed $t\in I$ and $g(0,\cdot)=\gamma_1$
\item $g(\cdot,t)$ is a continuous curve on a $W$ leaf for every fixed $t\in I$ and $g(\cdot, 0)=\gamma_2$.
\end{enumerate}
\end{lemma}
\begin{proof}

Without loss of generality we could assume 
$\gamma_1: I\to W^s(f,x)$ has the form 
$$\gamma_1(0)=x, \gamma_1(1)=y, \gamma_1([\frac{i}{n},\frac{i+1}{n}])\subset W^s_{k(i)}(\gamma_1(\frac{i}{n}))$$
for some $k(i)$. 

Consider the paths $\gamma_1|_{[0,\frac{1}{n}]}$ and $\gamma_2$, since $W\oplus W_{k(1)}^s$ is topological a product of $W$ and $W^s_{k(1)}$, we could easily find a continuous map $g_1:I^2\to W\oplus W^s_{k(1)}(x)$ such that 
\begin{enumerate}
\item $g_1(0,\cdot)=\gamma_1'|_{[0,\frac{1}{n}]}, g_1(\cdot,0)=\gamma_2$
\item $g_1(t,\cdot)$ (resp. $g_1(t,\cdot)$) is a continuous curve on a $W^s_{k(1)}$ (resp. $W$) leaf for every $t\in I$.
\end{enumerate}

Similarly, consider the paths $\gamma_1|_{\frac{1}{n}, \frac{2}{n}}$ and $g_1(\cdot,1)$, since $W\oplus W^s_{k(2)}$ is topological a product of $W$ and $W^s_{k(2)}$, we could find a continuous map $g_2:I^2\to W\oplus W^s_{k(2)}(\gamma_1(\frac{1}{n}))$ such that 
\begin{enumerate}
\item $g_2(0,\cdot)=\gamma_1|_{[\frac{1}{n}, \frac{2}{n}]}, g_2(\cdot,0)=g_1(\cdot,1)$
\item $g_2(t,\cdot)$ (resp. $g_2(t,\cdot)$) is a continuous curve on a $W^s_{k(2)}$ (resp. $W$) leaf for every $t\in I$.
\end{enumerate}

Repeat the constructions above, by induction we could construct a sequence of map $g_i:I^2\to M$. Moreover we could easily glue them along the vertical side and get a continuous map $g:I^2\to M$ satisfies the conditions of Lemma \ref{lemma: toral case induction}.
\end{proof}
 Suppose that Lemma \ref{lemma: toral case induction} holds, then consider any three points $x\in M, y\in W^s(x), z\in W^u(x)$. We can pick an accessible path $\gamma: I\to W^u(x)$ such that $$\gamma(0)=x, \gamma(1)=z, \gamma([\frac{i}{n}, \frac{i+1}{n}])\subset W^u_{k(i)}(\gamma(\frac{i}{n}))$$
for some $k(i)$. Choose an arbitrary accessible path $\gamma'$ from $x$ to $y$, then apply Lemma \ref{lemma: toral case induction} with $\gamma_1=\gamma', \gamma_2=\gamma|_{[0,\frac{1}{n}]}, W=W^u_1$ we get a continous map $g_0:I^2\to M $ satisfying all the conditions in Lemma \ref{lemma: toral case induction}. Now we apply Lemma \ref{lemma: toral case induction} again with $\gamma_1=g_0(1,\cdot), \gamma_2=\gamma|_{[\frac{1}{n},\frac{2}{n}]}, W=W^u_2$ we get a continuous map $g_1:I^2\to M$. Repeat the arguments above, by induction we could construct a sequence of map $g_i$, we can easily glue them along the vertical side and get a continuous map $g:I^2\to M$ such that $g(0,0)=x, g(0,1)=y, g(1,0)=z$ and $g(t,\cdot)$ (resp. $g(\cdot, t)$) is a continuous curve on a (un)stable leaf for each $t$. By Brin's result \cite{B77}, $f$ has a global product structure. 


Now we prove $f$ has a \emph{polynomial} global product structure. The first step is the following estimate. 
\begin{lemma}\label{lemma: tor cas Hol est}For any $C>0, n>0$ large enough, we have for any $x\in M$, any two coarse Lyapunov foliations $W_1, W_2$, $y\in W_2(x)$, $d_{W_2}(x,y)=R>0$, if we denote the holonomy along $W_2$ between $W_1$ leaves by $h$, then $$\|dh_{xy}\|\leq C(\max(\ln R, 0)^n+1)$$
\end{lemma}
\begin{proof}Without loss of generality we could work on $\tM$. By smoothness of coarse Lyapunov foliations, there exists $\epsilon>0$ such that for any $x\in M$, $y\in W_2(x)$, $d_{W_2}(x,y)\leq \epsilon$, we have $\|dh_{xy}\|\leq 2$.

Now we pick a norm one element $b\in \RR^k$ in the hyperplane corresponding to $W_1$ and uniformly contracts $W_2$. We denote by $T\in \ZZ^+$ the minimal non negative integer number such that $d_{W_2}(\tal(Tb)\cdot x, \tal(Tb)\cdot y)<\epsilon$, then 
\begin{equation}\label{eqn: est time}
T< \begin{cases}\frac{\ln R-\ln \epsilon}{\ln \lambda}+C_1 ~~\text{if $R>\epsilon$}\\
1~~ \text{if $R\leq \epsilon$}\end{cases}
\end{equation}
where $\lambda>1$ corresponds to the contracting rate of $b$ on $W_2$ and $C_1$ is a uniformly bounded error term coming from the rate of different metrics.

Notice that $$dh_{xy}=D\tal(-Tb)|_{E_1(\tal(Tb)\cdot y)}\circ dh_{\tal(Tb)\cdot x, \tal(Tb)\cdot y}\circ D\tal(Tb)|_{E_1(x)}.$$ By Theorem \ref{thm: osl spltt cont} we know $$\|D\tal(Tb)|_{E_1(x)}\|, \|D\tal(-Tb)|_{E_1(\tal(Tb)\cdot y)}\|\leq {C'}(T+1)^{n'}$$ for some $C', n'>0$
Therefore 
\begin{eqnarray*}
\|dh_{xy}\|&\leq& 2C'^2(T+1)^{2n'}\\
&\leq & C(\max(\ln R, 0)^n+1) ~~~\text{for some $C,n$ large enough (by \eqref{eqn: est time})} 
\end{eqnarray*}
\end{proof}
The following is a polynomial estimate within $W^s$, similar estimate also holds for $W^u$.
\begin{lemma}\label{lemma: poly est within Ws}For any $C>0, l>0$ large enough we have that, for  any $x\in M, y\in W^s(x)$ with $d_{W^s}(x,y)=R$, there exists a piecewise $C^1-$path $\gamma:I\to W^s(x)$ such that 
\begin{equation}\label{eqn: path cond}\gamma(0)=x, \gamma(1)=y, \gamma([\frac{i}{n}, \frac{i+1}{n}])\subset W^s_i
\end{equation} and each segment $\gamma([\frac{i}{n}, \frac{i+1}{n}])$ has length less than $C(R^l+1)$.
\end{lemma}

\begin{proof}By smoothness of foliations and Lemma \ref{lemma: prod Eisu} we know there exist $\epsilon>0$ such that for any geodesic $\gamma_0$ in $W^s$(with respect to the metric $d_{W^s}$) with length less than $\epsilon$, there exists $C_1>0$ such that we could find a continuous path $\gamma_1$ with $\gamma_1(0)=\gamma_0(0), \gamma_1(1)=\gamma_0(1)$ and satisfies the same condition in \eqref{eqn: path cond}, moreover each segment $\gamma_0([\frac{i}{n}, \frac{i+1}{n}])$ has length smaller than $C_1\cdot \text{length}(\gamma_0)$.

Without loss of generality we could assume $R$ in the statement of Lemma \ref{lemma: poly est within Ws} larger than $\max(\epsilon,1)$ (otherwise we can pick $\gamma$ and $C$ large such that $\text{length}(\gamma([\frac{i}{n}, \frac{i+1}{n}]))<\frac{C}{n}$) and we denote by $T$ the smallest integer number such that $d_{W^s}(f^n(x), f^n(y))<\epsilon$. Then as in the proof of Lemma \ref{lemma: tor cas Hol est}, we have $T<\frac{\ln R-\ln \epsilon}{\ln \lambda}+C_2$ for some $C_2$ large enough and $\lambda>1$ corresponds to the slowest contracting rate of $df|_{W^s}$.

Now by previous discussion we know there exists a piecewise $C^1$ path $\gamma_1$ such that $\gamma_1(0)=f^T(x), \gamma_1(1)=f^T(y)$, $\gamma_1([\frac{i}{n}, \frac{i+1}{n}])\subset W^s_i$ and $\text{length}(\gamma_1([\frac{i}{n}, \frac{i+1}{n}]))\leq C_1\epsilon$, then $f^{-T}(\gamma_1)$ satisfies all conditions in \eqref{eqn: path cond}, and \begin{eqnarray*}
\text{length}(f^{-T}(\gamma_1))&\leq& n\cdot C_1\epsilon\cdot \|df^{-1}\|^T\\
&\leq &O(e^{(\frac{\ln R-\ln \epsilon}{\ln \lambda}+C_2)\cdot \ln \|df^{-1}\|})\\
&\leq &O(R^{\frac{\ln \|df^{-1}\|}{\ln \lambda}})\\
&\leq & C\cdot R^l ~~\text{if $C, l$ large enough.}
\end{eqnarray*} \end{proof}

Now we come back to the proof that $f$ has polynomial global product structure.

\begin{lemma}\label{lemma: pgps for f}There exists $C>0$ large enough such that for any $x, y\in \hat{M}$ (the universal cover of $M$), if we denote $d_{\hat{M}}(x,y)$ by $R$, we have 
\begin{equation}\label{eqn: poly est on univ cover}
d_{W^s}(x,[x,y]), d_{W^u}([x,y],y)<C(R^2+1)
\end{equation}
\end{lemma}
\begin{proof}We will determine the largeness of $C$ later in the proof. Suppose \eqref{eqn: poly est on univ cover} is not true, then by compacticity of $M$ and continuity of foliations we could find $x,y\in \hat{M}$, such that $$d_{W^s}(x,[x,y]), d_{W^u}([x,y],y)\leq C(R^2+1)$$ and one of it is \textbf{equal to} $C(R^2+1)$, where $R:=d_{\hat{M}}(x,y) $. In addition we could assume that for any point $z$ on the geodesic $\gamma$ from $x$ to $y$, we have 
\begin{equation}\label{eqn: poly est z on cover}
d_{W^s}(z,[z,y]), d_{W^u}([z,y],y)\leq C(R^2+1)
\end{equation}

By continuity of $W^s$ and $W^u$, for $n$ large enough, there exists a constant $C_0$ does not depend on $\gamma$ and $n$ satisfies the following properties: we denote $\gamma(\frac{i}{n})$ by $x_i$, then for any $i$ we have 
\begin{equation}\label{eqn: geodesic devide}
\max(d_{W^s}(x_i, [x_i, x_{i+1}]),d_{W^u}([x_i,x_{i+1}], x_{i+1}))\leq C_0d_{\hat{M}}(x_i, x_{i+1})
\end{equation} 
Therefore 
\begin{equation}\label{eqn: sum of s path}
\sum_{i}d_{W^s}(x_i,[x_i,x_{i+1}])\leq C_0\cdot R
\end{equation}

The key to prove \eqref{eqn: poly est on univ cover} is to control the size of $d_{W^s}([x,x_i], [x,x_{i+1}])$. Denote the global holonomy along $W^u$ foliation between two $W^s$ leaves on $\hat{M}$ by $h^u$. Consider the geodesic $\gamma^s_i$ with respect to $d_{W^s}$ from $x_i$ to $x_{i+1}$, then 
\begin{equation}\label{eqn: size est xi,i+1}
d_{W^s}([x,x_i], [x,x_{i+1}])\leq \text{length}(h^u_{x_i,[x,x_i]}\cdot\gamma^s_i)\leq d_{W^s}(x_i, [x_i,x_{i+1}])\cdot \sup_{z\in \gamma^s_i} \|dh^u_{z, [x,z]}\|
\end{equation}

By \eqref{eqn: poly est z on cover} we know, without loss of generality, we can assume $n$ and $R$ are large enough \footnote{For $R$ not big, we could choose $C$ in \eqref{eqn: poly est on univ cover} large enough such that \eqref{eqn: poly est on univ cover} holds.} such that for any $z\in \gamma^s_i$, cf. Lemma 2 in \cite{B77}, \begin{equation}\label{eqn: est d z, [x,z]}
d_{W^u}(z,[x,z])\leq 2C(R^2+1)
\end{equation}
Then by the unstable version of Lemma \ref{lemma: poly est within Ws}, there are constants $C_1, l_1>0$ such that for any $z\in \gamma^s_i$ there is a piecewise $C^1$ unstable path $\gamma^u_z:I\to W^u(z)$ such that 
\begin{equation}\label{eqn: devide z, [x,z]}
\gamma^u_z(0)=z, \gamma^u_z(1)=[x,z], \gamma^u_z([\frac{i}{k}, \frac{i+1}{k}])\subset W^u_i
\end{equation}
and for each $i$ we have 
\begin{equation}\label{eqn: est devid gamma z u}
\text{length}(\gamma^u_z([\frac{i}{k}, \frac{i+1}{k}])\leq C_1(R^{l_1}+1)
\end{equation}

Therefore there exists $C_2$ and $l_2$ large enough which do not depend on the choices of $x,y, x_i, z$ such that for any $z\in \gamma_i^s$, any coarse Lyapunov distribution $E^s_j$ within $W^s$, we have
\begin{eqnarray}\label{eqn: est dhu z [x,z] on Ej}
\|dh^u_{z,[x,z]}|_{E^s_j}\|&=&\|dh^u_{\gamma^u_z(\frac{k-1}{k}), \gamma^u_z(1)}\circ \cdots \circ dh^u_{\gamma^u_z(0), \gamma^u_z(\frac{1}{k})}|_{E^s_j}\|  \\\nonumber
&\leq & \prod_{i=1}^k \|dh^u_{\gamma^u_z(\frac{i-1}{k}), \gamma^u_z(\frac{i}{k})}|_{E^s_j} \|~~\text{we use joint integrability here}\\ \nonumber
&\leq & C_2 (\max(\ln R,0)^{l_2}+1) ~~~~\text{by Lemma \ref{lemma: tor cas Hol est} and \eqref{eqn: est devid gamma z u}}\\\nonumber
&\leq &C_2((\ln R)^{l_2}+1)~~~~\text{(we can assume here $R>1$)}
\end{eqnarray}
By continuity of coarse Lyapunov distributions we know the angles between any pair of subspaces $$\oplus_{j\in J}E^s_j, \oplus_{j\in J'}E^s_{j}, J\cap J'=\emptyset$$ is uniformly bounded away from $0$. Therefore by \eqref{eqn: est dhu z [x,z] on Ej} we know there exists $C_3>0$ which does not depend on the choice of $x,y,z$ such that 
\begin{equation}\label{eqn: est final dhu z [x,z]}
\|dh^u_{z,[x,z]}\|\leq C_3(\ln R)^{l_2}+1)
\end{equation}
Compare \eqref{eqn: est final dhu z [x,z]} with \eqref{eqn: size est xi,i+1}, \eqref{eqn: sum of s path} we have 
\begin{eqnarray*}
&&\sum_{i}d_{W^s}([x,x_i], [x,x_{i+1}])\\
&\leq & \sum_i d_{W^s}(x_i, [x_i,x_{i+1}])\cdot \sup_{z\in \gamma^s_i} \|dh^u_{z, [x,z]}\| ~~\text{by \eqref{eqn: size est xi,i+1}}\\
&\leq &  C_3((\ln R)^{l_2}+1)\cdot \sum_i d_{W^s}(x_i, [x_i,x_{i+1}])~~\text{by \eqref{eqn: est final dhu z [x,z]}}\\
&\leq & C_3C_0\cdot R((\ln R)^{l_2}+1)~~\text{by \eqref{eqn: sum of s path}}\\
&\leq & C(R^{\frac{3}{2}}+1)~~\text{for $C$ large}
\end{eqnarray*} 
As a corollary, we have that $$d_{W^s}(x,[x,y])\leq \sum_{i}d_{W^s}([x,x_i], [x,x_{i+1}])<C(R^2+1)$$
here without loss of generality we assume that $R>1$. By the same method, we can also prove that $d_{W^u}([x,y],y)<C(R^2+1)$, which contradicts with our choice of $x,y$. In summary we could find $C$ large enough such that \eqref{eqn: poly est on univ cover} holds on $\hat{M}$.
\end{proof}
\end{proof}

\section{Resonance free case}\label{sec:res-free}
In this section we prove Theorem \ref{main resonance free case}. Passing to a finite cover and taking restriction to a finite index subgroup of $\ZZ^k$ if necessary we could assume all the properties in Theorem \ref{thm: osl spltt cont} holds for $M$ and $\al$. Resonance free condition and TNS condition implies there is no proportional Lyapunov functionals. Therefore coarse Lyapunov splitting of $D\al$ on $TM$ is actually Oseledec splitting. By Theorem \ref{thm: osl spltt cont} we know $D\al$ has polynomial deviation with each coarse Lyapunov distribution.

\subsection{Smoothness of Lyapunov foliations}\label{sec: smth fol res free}
The first step is the smoothness of coarse Lyapunov distribution. The basic idea is to use $C^r$ section theorem in \cite{HPS} and Weyl chamber picture for resonance free action, which is similar to  \cite{KS07} and section 5.3 of \cite{DX}. For later use we give an outline here.

We consider a generic two-dimensional subspace $P$ in $\RR^k$ such that $P$ intersects each Lyapunov hyperplanes along distinct lines. In addition, since $\al$ is resonance free, $P$ can be chosen such that for any $b\in \ker \chi_1\cap P-\{0\}, \chi_i(b)\neq \chi_j(b)$ for any $\chi_i\neq \chi_j$. For any Lyapunov functional $\chi_j$, we denote by $E_j$ the corresponding Lyapunov distribution, $\H_j$ the Lyapunov half space such that $\chi_j$ on $\H_j$ is negative. $H_j:=\H_j\cap P$ is the half-plane in $P$. Now we reorder these halfplanes counterclockwisely such that $H_1$ is the half space corresponding to $E^1$. Then by TNS condition there exists a unique $i > 1$ such that $$\cap_{1\leq j\leq i}H_j\cap\cap_{j'>i}-H_{j'}\neq\emptyset$$

For any element $a\in \cap_{1\leq j\leq i}-\mathcal{H}_j\cap\cap_{j'>i}\mathcal{H}_{j'}$, by our assumption of $i$, $$\oplus_{1\leq j\leq i} E_j=E^u_a$$ Then  $\oplus_{1\leq j\leq i} E_j$ is uniformly $C^{\infty}$ along $W^u_a$  and in particular along $W^1$.

We choose a unit vector $b\in \ker\chi_1\cap P$ such that $b\in H_j$ for any $2\leq j\leq i$. By our choice of $P$ we know for any $j',j''\geq l+1$, $\chi_{j'}(b)\neq \chi_{j''}(b)$. Therefore we could reorder the indices $1,\dots, i$ by $j_i, j_{i-1},\cdots, j_1$ such that $$\chi_{j_i}(b)<\cdots <\chi_{j_{2}}(b)<\chi_{j_1}(b)=0$$

We consider the following corollary of $C^r-$section theorem in \cite{HPS}, or cf. Corollary 5.6 in \cite{DX}.

\begin{lemma}\label{coro: coro cr section}Let $f$ be a $C^{\infty}$ diffeomorphism of a compact smooth manifold $M$. Let $W$ be an $f-$invariant topological foliation with uniformly $C^{\infty}-$leaves and  $\|Df |^{-1}_{TW(x)}\|:= \al_x$ for all $x\in M$. Let $E^1$ and $E^2$ be continuous f-invariant distributions on $M$ such that the distribution $E = E^1\oplus E^2$ is uniformly $C^\infty$ along $W$ and $E^1\oplus E^2$ is a dominated splitting in the sense that for any $x \in M$,
$$k_x:=\frac{\max_{v\in {E}^2(x), \|v\|=1}\|Df(v)\|}{\min_{v\in {E}^1(x), \|v\|=1}\|Df(v)\|}<1$$
If $\sup_{x\in M}k_x\al_x^r<1$.
Then $E^1$ is uniformly $C^r$ along the leaves of $W$.
\end{lemma}

We consider an arbitrary $m$, $1<m< i$ and apply Lemma \ref{coro: coro cr section} to $f=\tilde{\al}(b), W=W^1, E^1=\oplus_{s=1}^m E_{j_s}, E^2=\oplus_{s=m+1}^i E_{j_s}$. Notice that by polynomial deviation estimate in Theorem \ref{thm: osl spltt cont} we have
\begin{eqnarray*}
\|D\tilde{\al}(b)|^{-1}_{TW(x)}\|=\|D\tilde{\al}(b)|^{-1}_{E^{\chi_1}(x)}\|&\leq& O(\|b\|^L) \text{ for $\|b\|$ large and some $L>0$.}\\
\|D\tilde{\al}(b)(v)\|&\geq& O(e^{\chi_{j_m}(b')})\text{ for any  unit vector $v\in E_1$.}\\
\|D\tilde{\al}(b)(v)\|&\leq& O(\|b\|^L\cdot e^{\chi_{j_{m+1}}(b)})\\
&&\text{ for any $v\in E_2, \|v\|=1$ and some $L>0$.}
\end{eqnarray*}
Then take $b$ such that $\|b\|$ large enough, by Lemma \ref{coro: coro cr section} (if necessary we could replace $b$ by $nb$ for $n$ large) we know $E^1=\oplus_{s=1}^m E_{j_s}$ is uniformly $C^{\infty}$ along the leaves of $W^1$ for any $m<i$.

Similarly by considering $-b$, we have  $\oplus_{s=m+1}^i E_{j_s}$ is uniformly $C^\infty$ along the leaves of $W^1$ for any $m<i$. Therefore by taking intersection, $E_m$ is uniformly $C^{\infty}$ along $W^1$ for any $m\leq i$.

Consider the halfplanes $\{-H_l\}$, mimick the proof above we get for any $j\geq i$, $E_j$ is uniformly $C^{\infty}$ along $W^1$. The same proof holds for any $W^k$. Then by Journ\'e Lemma we get the smoothness of $E_i$.

\subsection{Construction of a smooth invariant connection}
Now we claim that there is a global $\al-$invariant smooth connection on $M$. If it holds then by a rigidity result by Benoist, Labourie \cite{BL} we could conclude that $\al$ is smoothly conjugate to an affine action on an infranilmanifold. Then as in the section 3.3 in \cite{KS07} we can easily prove that this infranilmanifold is finitely covered by a torus.

Firstly we list several useful properties of non-stationary linearization for smooth diffeomorphisms on contracting foliations, under a pointwise $1/2-$pinching assumption, see \cite{S05, KS06} for more details.

Let $f$ be a $C^\infty$ diffeomorphism of a compact smooth manifold $M$, and let $W$ be
a continuous invariant foliation with $C^\infty$ leaves which is contracted by $f$, i.e.
$\|Df|_{TW}\| < 1$ in some Riemannian metric. Moreover we assume that there exist $C > 0$ and  $\gamma<1$ such that for all $x\in M$ and $n\geq 0$,
$$\|(Df^n|_{T_xW})^{−1}\|\cdot \|Df^n|_{T_xW}\|^2 ≤ C\gamma^n.$$
Then for every $x\in M$ there exists a $C^\infty$ diffeomorphism $\H_x : W_x\to T_xW$ such that
\begin{enumerate}
\item $\H_{fx}\circ f \circ \H_x^{-1}= Df|_{T_xW}$,
\item $\H_x(x) = 0$ and $D_x\H_x$ is the identity map,
\item $\H_x$ depends continuously on $x\in M$ in $C^\infty$ topology.
\item Such a family $\H_x$ is unique and depends smoothly on $x$ along the leaves of $W$,
\item The map $\H_y \circ \H_x^{-1} : T_xW\to T_yW$ is affine for any $x\in M$ and $y\in W_x$. Hence
the non-stationary linearization $\H$ defines affine structures on the leaves of $W$.
\end{enumerate}
Then the non-stationary linearization $H$ defines affine structures on the leaves of $W$, Moreover it is a linearization for any diffeomorphism $g$ \textbf{commutes with} $f$, i.e. $$\H\circ g= Dg\circ \H$$

As mentioned above, each Lyapunov distribution is actually a coarse Lyapunov distribution hence integrable. By Theorem \ref{thm: osl spltt cont} we know for any Lyapunov distribution $E$ and its integral Lyapunov manifold $W$, for any $a\in \ZZ^k$ contracts $E$, the cocycle $D\al(a)|_{E}$ satisfies pointwise $\frac{1}{2}-$pinching condition above, there therefore there is a unique $\al-$invariant affine structure on $W$, i.e. for every $x\in M$ there exists a $C^\infty$ diffeomorphism $\H_x : W_x \in T_xW$ satisfies properties of non-stationary linearization we mentioned above.

Now we consider the leafwise flat connection $\nabla_E$ along $W$ induced by the affine structure on $W$. Obviously the parallel transport induced by $\nabla_E$ on $E$ over $W$ is exactly the parallel transport induced by the affine structure: $\H_y\circ \H_x^{-1}$ for any $x\in M$ and $y\in M$. A priori $\nabla_E$ is uniformly smooth along $W$ but only continuous on ambient $M$. 
\begin{lemma}\label{lemma: r-f case smooth conn}$\nabla_E$ is smooth on $M$.
\end{lemma}
\begin{proof}Our strategy to prove the smoothness of $\nabla_E$ is the following: a priori $\nabla_E$ is only defined leafwisely, i.e. for a path $\gamma$ not contained in $W$, the parallel transport $\Gamma(\gamma)$ of connection $\nabla_E$ is not well-defined. We will extend $\nabla_E$ to a new smooth connection $\tilde{\nabla}_E$ of the bundle $E$ over ambient $M$ by considering suitable Holonomies. Therefore $\nabla_E$ as the restriction of $\tilde{\nabla}_E$ on $W$ is also smooth on $M$.

We consider the parallel transport $H_{xy}:=\H_y\circ \H_x^{-1}$ induced by $\nabla_E$. We slightly abuse the notation $H_{xy}$ here because the parallel transport in fact coincides with the Holonomy of the fiber-bunched cocycle $D\al(a)$ along $W$. \footnote{Since $W$ subfoliates $W^s_a$, we could freely define the Holonomy map along $W^s_a$, hence $W$.} The reason is that for any $a$ contracts $E$, for any $x\in M, y\in W_x$ 
$$H_{xy}=D\al(-na)\circ H_{\al(na)\cdot x, \al(na)\cdot y}\circ D\al(na)$$
and $H$ is uniformly $C^1$ along $W$. Therefore along $W$ we know $H_{xy}$ coincides with the definition formula in Proposition \ref{prop: properties holonomies} ($H_{\al(na)\cdot x, \al(na)\cdot y}$ can be viewed as the linear identification from $E_{\al(na)\cdot x}$ to $E_{\al(na)\cdot y}$ which uniformly H\"older close to the identity).

For fixed $E, W$, we apply Proposition \ref{prop: E12 splt} to $E$, then we get the associated splitting $TM=E_1\oplus E\oplus E_2$, foliations $W_1, W_2, W\oplus W_1, W\oplus W_2$ and elements $c_i$ satisfy all the conditions in Proposition \ref{prop: E12 splt}. Then we have 
\begin{lemma}\label{lemma: smth Hxy W1,2}The parallel transport $H_{xy}: E_x\to E_y, y\in W_x$ is uniformly smooth when  $x,y$ moves along $W_1$. Similar property holds for $W_2$.
\end{lemma}
\begin{proof}For fixed $x\in M, y\in W_x$, we take arbitrary $x'\in W_1(x), y'\in W_1(y) $ such that $y'\in W_{x'}$. Pick $c_i$ as in Proposition \ref{prop: E12 splt} and sufficiently close to the hyperplane $L$ corresponds to $E$. As in section \ref{sec: unif bunched, Holo} we denote by $H^j_{c_i, x,y}$ the stable Holonomy for cocycle $D\al(c_i)|_E$ along $W_j$, $h^i$ the holonomy map along $W_i$ between two $W$ leaves. And we denote by $H^s_{c_i,x,y}$ the stable Holonomy for $D\al(c_i)|_E$ along $W^s_{c_i}$. Then by Lemma \ref{lemma: Hlmy 2nd def} and uniqueness of Holonomy for cocycles (Proposition \ref{prop: properties holonomies}) we have that 
\begin{eqnarray*}
H_{x'y'}&=&H^s_{c_1,x',y'}~~(\text{$H$ coincides with the Holonomy of $D\al(a)$ along $W$})\\
&=&H^s_{c_1,y,y'}\circ H^s_{c_1, x,y}\circ H^s_{c_1,x',x}~~(\text{since }x',y',x,y\in W^s_{c_1})\\
&=&H^1_{c_1,y,y'}\circ H_{xy}\circ H^1_{c_1,x',x}~~(\text{by definition})\\
&=&dh^1_{yy'}\circ H_{xy}\circ dh^1_{x'x}~~(\text{by Lemma \ref{lemma: Hlmy 2nd def}})
\end{eqnarray*}
But $dh^1_{yy'}, dh^1_{x'x}$ (uniformly) smoothly depend on $x',y'$ for fixed $(x,y)$, therefore we get that $H_{xy}$ depends on $(x,y)$ (uniformly) smoothly when $x,y$ moves along $W_1\times W_1$.
\end{proof}



By the proof of Lemma \ref{lemma: smth Hxy W1,2}, we know $H^s_{c_1,x,y}, x,y\in W\oplus W_1$ (or $W$) actually induces the smooth integrable horizontal sections $H_{0,1}$ ($H_0$ resp.) of bundle $E$ over $W\oplus W_1$ ($W$ resp.), i.e. for any $x\in M, v\in E_x$, the image of $H^s_{c_1,x,\cdot}(v)$ on $W\oplus W_1$ is the horizontal integrable section $H_{0,1}$ passing through $v$. For smoothness, $H_0$ in fact defined from the parallel transport $H$ induced by $\nabla_E$, therefore $H_0$ is smooth along $W$. For $H_{0,1}$, notice that  $H^s_{c_1}$ coincides with $dh^1$ ($H$) along $W_1$ ($W$ respectively). Then by Journ\'e Lemma we know the image of $H^s_{c_1,x,\cdot}(v)$ on $W\oplus W_1$ is smooth. Integrability is directly from the definition of Holonomy.

Similarly we could define the smooth integrable horizontal sections $H_{0,2}$ ($H_2$) induced by $H^s_{c_2}$  of bundle $E$ over $W\oplus W_2$ ($W_2$ respectively). By uniqueness of Holonomy (Proposition \ref{prop: properties holonomies}) we know $H_{0,2}$ coincides with $H_{0,1}$ and $H_0$ along $W$. We denote by $\tE, \tilde{E}_{0,1}, \tilde{E}_2$ the tangent spaces of $H_0, H_{0,1}, H_2$ respectively. Suppose $E_1=\oplus E_{1i}, E=\oplus E_{2j}$, where $E_{1i}, E_{2j}$ are coarse Lyapunov distributions and $W_{1i}, W_{2j} $ are corresponding Lyapunov foliations. As above we could define the smooth integrable horizontal sections $H_{1i}$ ($H_{2j}$) induced by $H^s_{c_1}$ ($H^s_{c_2}$ resp.) over $W_{1i}$ ($W_{2j}$ resp.). We denote by $\tE_{1i}, \tE_{2j}$ the tangent spaces of $H_{1i}, H_{2j}$ respectively. Then $\tE_{0,1}=\tE\oplus\oplus_i \tE_{1i}, \tE_2=\tE_{2j}$.

Our goal is to prove the following lemma,
\begin{lemma}\label{lemma: TNS PCF arg inte}The distributions $\tilde{E}, \tE_{1i} \tilde{E}_{2j}$ are smooth on $M$.
\end{lemma} 
If Lemma \ref{lemma: TNS PCF arg inte} holds, since the parallel transport of $\nabla_E$ along $W$ is in fact the translation through $H_{0}$, therefore we get the smoothness of $\nabla_E$ from the smoothness of $\tE$ ......$H_{0}$.
\begin{proof}[Proof of Lemma \ref{lemma: TNS PCF arg inte}] In fact the proof here is similar to the proof of Proposition 5.1 in \cite{DX} and discussion in section \ref{sec: smth fol res free}. We view $D\al|_E$ as a skew product system over $\al$ on $M$ where fiber is $E$. $H_0, H_{1i}, H_{2j}$ are $D\al|_E-$invariant foliations which can be viewed  as the lifts of the foliations $W, W_{1i}, W_{2j}$ from the base to the fiber bundle.

So our goal is to prove the smoothness of the $D\al|_E-$invariant foliations $H_0,H_{1i}, H_{2j}$ or their tangent spaces $\tE,\tE_{1i}, \tE_{2j}$. Notice that these foliations are induced from Holonomies. By linearity of Holonomies we know 
\begin{itemize}
\item[\textbf{Fact 1}]$H_0,H_{1i}, H_{2j}$ (hence $\tE,\tE_{1i}, \tE_{2j}$) are smooth along the fiber $E$.
\end{itemize}

By Theorem \ref{thm: osl spltt cont} we know up to a subexponential factor bounded by some polynomial, $\al(E)$ contracts (or expands) $W, W_{1i}, W_{2j}$ (hence $E, E_{1i}, E_{2j}$) with uniform exponential speed. The contracting rate is exactly the Lyapunov functional associated to the foliation. By property (4). of Holonomies in Proposition \ref{prop: properties holonomies}, it is easy to see that $D\al|_E$ also contracts (or expands) $H_0$ with the same rate as $W$. Similar properties also hold for $H_{1i}, H_{2j}$. In particular,

\begin{itemize}
\item[\textbf{Fact 2}]For $\ZZ^k-$action $D\al|_E$ on the fiber bundle $E$ over $M$, the associated Lyapunov distributions are $\tE,\tE_{1i},\tE_{2j}$. The associated Lyapunov functionals of $\tE$ (or $\tE_{1i}, \tE_{2j}$ resp.) are the same as the Lyapunov functionals associated to $E$ (or $E_{1i}, E_{2j}$ resp.) of $\al$, moreover for $D\al|_E$ we have similar \emph{polynomial deviation} property as in Theorem \ref{thm: osl spltt cont} on $\tE,\tE_{1i},\tE_{2j}$.
\end{itemize}

Now we only need to prove that \begin{itemize}
\item[\textbf{Claim}] $\tE$ (or $\tE_{1i}$, $\tE_{2j}$) is uniformly smooth along $H_0, H_{1i'}, H_{2j'}$ for any $i,j, i', j'$.
\end{itemize}
If the claim is true, then by inductive application of Journ\'e Lemma we know $\tilde{E}_{1i}, \tilde{E}_{2j},\tilde{E}$ are smooth distributions since they are smooth along the vertical fiber $E$.

As in section \ref{sec: smth fol res free}, we considering a generic plane $P$ in $\RR^k$ such that $P$ intersects each Lyapunov hyperplanes along distinct lines and for any $b\in \ker \chi_1\cap P-\{0\}, \chi_i(b)\neq \chi_j(b)$ for any $\chi_i\neq  \chi_j$; we could define $\H_j,H_j, \chi_j$ similarly; moreover we could reorder these halfplanes counterclockwisely as in section \ref{sec: smth fol res free}. In particular, we rewrite the distributions in  $\tE, \tE_{1i},\tE_{2j}$ and the associated foliations $H_0, H_{1i}, H_{2j}$ by new symbols
\begin{equation}\label{eqn: nota redef}
L_1,\dots, L_j,\dots;~~F_1,\dots, F_j,\dots
\end{equation}
where $L_j$ ($F_j$) is the distribution (foliation) in $\tE, \tE_{1i},\tE_{2j}$ ($H_0, H_{1i}, H_{2j}$) corresponding to the halfplace $\H_j$.

Then by TNS condition again there exists unique $i$ such that 
$$\cap_{1\leq j\leq i}H_j\cap\cap_{j'>i}-H_{j'}\neq\emptyset$$

Again we pick any element $a\in \cap_{1\leq j\leq i}H_j\cap\cap_{j'>i}-H_{j'}$, then the unstable distribution of the cocycle $D\al(a)|_E$ is $\oplus_{1\leq j\leq i} L_j$. As a result, $\oplus_{1\leq j\leq i} L_j$ is  smooth along $\oplus_{1\leq i} F_j$ which is the integral foliation of $\oplus_{1\leq j\leq i}L_j$. In particular, $\oplus_{1\leq j\leq i} L_j$ is smooth along $F_1$.

Then as in section \ref{sec: smth fol res free} we could pick $b\in \cap_{2\leq j\leq i}H_j\cap \ker \chi_1\cap P$ and then further reorder the indices $1,\dots, i$ by $j_i,\dots, j_1$ such that $\chi_{j_i}(b)<\dots<\chi_{j_1}(b)=0$. Then by the same argument as in section \ref{sec: smth fol res free} we could prove that $\oplus_{s=1}^m L_{j_s}$ (thanks to polynomial deviation) is uniformly smooth along the leaves of $F_1$ for any $m<i$. Similarly by considering $-b$ we have $\oplus_{s=m+1}^iL_{j_s}$ is uniformly smooth along the leaves of $F_1$ for any $m<i$. Therefore by taking intersection, $L_m$ is uniformly smooth along $F_1$ for any $m\leq i$. By considering the halfplanes $\{-H_l\}$ we could get for any $j\geq i$, $L_j$ is uniformly smooth along $F_1$. The same proof also holds for any $F_k$. Then by Journ\'e Lemma we know each $L_j$ is a smooth distribution on the fiber bundle $E$ over $M$.
\end{proof}\end{proof}

We collect all of these $\nabla^i:=\nabla_{E^i}$ together for all Lyapunov distributions $E^i$ and define a new connection $\nabla$ on $TM$ in a canonical way:

$$\nabla_XY:=\sum_i \nabla^i_{X^i}Y^i+ \sum_{i\neq j} \Pi_j [X^i, Y^j] $$
where $\Pi_j$ is the projection onto $E^j$ with respect to the coarse Lyapunov splitting. Since $\nabla^i, E^i$ are $\al-$invariant and smooth, so is $\nabla$. Therefore for by the discussion in the beginning of this section we know a finite cover of $\al$ is smoothly conjugate to affine maps on a torus.


\appendix
\section{Proof of Case 2 of Lemma \ref{lemma: cont red fin cov}}
\begin{proof}
As in section 4.5 of \cite{KS13}, we consider the flag of subspaces $V^i$ spanned by the first $n_i=d_1+\cdots+d_i$ coordinate vectors in $\RR^d$, $i=1,\dots, l$. Here $d_i$ and $d$ are defined as case 1. Consider $G_\ast$ which is the stablizer of this flag in $G$. Then $G_\ast$ contains $G_0$ and therefore we could assume $G_\ast$ is of finite index $k$ in $G$. The orbit of this flag under $G$ consists of $k$ distinct flags in $\RR^d$ which we denote by 
\begin{equation}\label{eqn: distinct flags}
W^j=\{V^{j,1}\subsetneqq \cdots\subsetneqq V^{j,l}=\RR^d \},~~j=1,\dots, k
\end{equation}
Then any $g\in G$ permutes these flags and preserves their union. Without loss of generality we could assume $k\geq 2$.

For $i=1,\dots, l-1$, the subspaces $V^{j,i}$ have dimension $n_i$. Since some of them may coincide, so we denote by $k_i$ the number of distinct one. We set $U^{(i)}:=V^{1,i}\cup \cdots\cup V^{k,i}$. Then as in Case 1, $U^{(i)}$ corresponds to a $D\al-$invariant measurable family $\{\cU^{(i)}_x\}_{x\in M}$ where $\cU^{(i)}_x$ depends measurably on $x$ and is a union of $k_i$ distinct $n_i-$dimensional subspaces of $E_{m_n\chi}$.

As in Case 1, $\{\cU^{(i)}_x\}_{x\in M}$ could be identified with a $\wedge^{d_i}(D\al(c_4))|_{m_n\chi}-$invariant measure $m$ on the bundle $\PP(\wedge^{d_i}(E_{m_n\chi}))$; the cocycle $\wedge^{d_i}(D\al(c_4))|_{m_n\chi}$ is has coincide Lyapunov exponents within $\wedge^{d_i}(E_{m_n\chi})$. In addition $\wedge^{d_i}(D\al(c_4))|_{m_n\chi}$ admits \textbf{uniform} stable Holonomy along $W_2$ and \textbf{non-uniform} unstable Holonomy along $W\oplus W_1$. Then by Proposition \ref{prop: non unif inv prpl} $m$ is $\mu-$almost surely invariant under stable and unstable Holonomy. Therefore mimicking the proofs of Lemmas \ref{lemma: product structure and continuity}, \ref{lemma: hold cont cEi}, we know the family $\{\cU^{(i)}_x\}_{x\in M}$ coincides with a H\"older continuous family $\mu-a.e.$, which we still denote by $\cU^{(i)}$.

Without loss of generality we could assume $\al$ has a fixed point $q$ on $M$ (otherwise we consider the restriction of $\al$ on a finite index subgroup of $\ZZ^k$). As in section 4.5 of \cite{KS13}, there is a normal subgroup $H_q$ of $\pi_1(M,q)$ such that on the finite cover $\bar M$ of $M$  corresponding to $H_q$, the lift $\tilde\cU^{(i)}$ of $\cU^{(i)}$ extends to H\"older continuous $n_i-$dimensional subbundles $\tilde{\cU}^{1,i},\dots,\tilde{\cU}^{k_i,i}$ for each $i$. Moreover for any lift $\tilde{q}$ of $q$, there exists a unique lift $\bal$ of $\al$ on $\bar M$ such that $\tilde{q}$ is an $\bal-$fixed  point \footnote{In fact the argument here is just a special case of  section 4.5 in \cite{KS13}, which also works for general cocycles. Here we only consider the derivative cocycle.}. 

Obviously now $\bal$ preserves the union $
\tilde{\cU}^{(i)}$ of subbundles $\tilde{\cU}^{1,i},\dots,\tilde{\cU}^{k_i,i}$ for each $i$. By continuity of these subbundles there exists $N\in \ZZ^+$ such that the cocycle $D\bal|_{N\cdot \ZZ^k}$ preserves every subbundle $\tilde{\cU}^{j,i}$. Moreover by the discussion in \cite{KS13}, we could even arrange these subbundles $\tilde{\cU}^{j,i}$ into the $D\bal|_{N\cdot \ZZ^k}-$invariant flags $\tilde{\W}^{1},\dots,\tilde{\W}^{k}$ (which corresponding to $W^1,\dots W^k$ in $\RR^d$). \footnote{In \cite{KS13} to arrange $\tilde{\cU}^{j,i}$ we need the ergodicity of $\tilde{\mu}$ which is the lift of the measure on the finite cover. In our setting the ergodicity is coming from that $\tilde{\mu}$ is the lift of Bowen-Margulis measure of an Anosov diffeomorphism hence it is still a Bowen-Margulis measure of the lifted diffeomorphism (since holonomy-invariance of conditional measure on (un)stable foliations is a local property). Then (by Hopf-argument) the Bowen-Margulis measure should be ergodic.}  We fix one of these flags and denote it by $$\tilde{\W}: \{0\}=\cE_0\subsetneqq \cE_1\subsetneqq\cdots \subsetneqq\cE_l=\bar{E}_{m_n\chi}$$
where $\bar{E}_{m_n\chi}$ is the lift of $E_{m_n\chi}$ on $\bar M$. 

It remains to show in each factor bundle $\cE_i/\cE_{i-1}$ there is a continuous conformal structure invariant under the cocycle induced by $D\bal|_{N\cdot \ZZ^k}$. The proof is the same as in \cite{KS13}. Recall that before lifting to $\bar M$, up to a measurable coordinate change the cocycle $D\al|_{N\cdot \ZZ^k}$ is taking values in $G\subset GL(d,\RR)$. $G_\ast, V_i$ are defined as the beginning of this section. Any matrix $A$ in $G_0$ preserves the subspaces $V_i$ and the standard conformal structure $
\sigma$ on $\RR^{d_i}\cong V^i/V^{i-1}$. The orbit of $\sigma$ under the stablizer $G_\ast$ of the flag is a finite set $\mathcal{O}$ in the space of conformal structures on $\RR^{d_i}$ (for more details for the space of conformal structures cf. \cite{T86}, \cite{KS13}). Since $G_\ast$ preserves $\mathcal{O}$, then $G_\ast$ preserves the center $\sigma_\ast$ of the smallest ball containing $\mathcal{O}$, which is unique in the space of conformal structures. We push $\sigma_\ast$ by the action of $G$. Then we obtain the conformal structures $\sigma_j$ on the factors $V^{j,i}/V^{j,i-1}$ of the flags $W^j$. Therefore for any $g\in G$, if $g$ maps $W^j$ to $W^{j'}$ then the factor map $\bar g: V^{j,i}/V^{j,i-1}\to V^{j',i}/V^{j',i-1}$ takes $\sigma_j$ to $\sigma_{j'}$. By the construction of $\tilde\W$ we know (up to a measurable coordinate change) $\tilde\W$ corresponds to one of the flag $W^j$ almost everywhere. Therefore $\sigma_j$ on $V^{j,i}/V^{j-1,i}$ corresponds to a measurable $D\bal|_{N\cdot \ZZ^k}-$invariant conformal structure $\tau$ on $\cE_{i}/\cE_{i-1}$. As in Case 1, $\tau$ could be identified with a probability measure on $\PP(\cE_i/\cE_{i-1})$ which is invariant under the action induced by $D\bal|_{N\cdot \ZZ^k}$. Since the covering map is a local isometry, we have that $D\bal(Nc_4)|_{\cE_i/\cE_{i-1}}$ has coincide Lyapunov exponents in $\cE_i/\cE_{i-1}$ with respect to $\tilde{\mu}$ (the lift of $
\mu$ on $\bar M$). Moreover we could assume that $D\bal(Nc_4)|_{\cE_i/\cE_{i-1}}$ admits uniform stable Holonomy along $W_2$ and non-uniform unstable Holonomy $W\oplus W_1$ as in Case 1 (if necessary we could pick $c_4$ sufficiently close to the Weyl chamber wall). Therefore by previous redefining argument we know $\tau$ coincides with a H\"older continuous conformal structure on $\cE_i/\cE_{i-1}$ $\tilde{\mu}-$a.e. which we still denote by $\tau$, obviously $\tau$ is $D\bal|_{N\cdot 
\ZZ^k}$ invariant, which completes the proof of Lemma \ref{lemma: cont red fin cov}. 

In summary, our proof for Case 2 is basically the same as section 4.5 in \cite{KS13}. The main difference is that here we use our non-uniform redefining argument (rather than the uniform one in \cite{KS13}) to prove the existence of H\"older continuous extension of the dynamically-defined objects.
\end{proof}
 
\section{Proof of Lemma \ref{lemma: coh eqn closing}}\label{sec: app coh eqn cls}
\begin{proof}The proof is basically the same as section 3.2 of \cite{KS07}. For completeness we give a proof here. By transitivity of $\{\tal(ta_0)\}$, we pick a point $x^\ast$ with dense orbit $\mathcal{O}^\ast=\{\tal(ta_0)\cdot x^\ast\}$. We choose $$\phi(\tal(ta_0)\cdot x^\ast):=q(x^\ast, ta_0)^{-1}$$
Then by the construction we know $\phi$ satisfies \eqref{equation: coho eqn gene sing a} on $\mathcal{O}^\ast$. Then as in section 3.2 of \cite{KS07}, to extend $\phi$ on whole $\tM$ in a H\"older continuous way, we only need to prove there exist $\epsilon_0, K>0$ such that for any $x\in \mathcal{O}^\ast, t>0$, if $\dist(x, \tal(a)\cdot x)<\epsilon_0$ for some $x\in \tM$ and $a=ta_0$ with $t>0$, then 
\begin{equation}\label{eqn: KS07 3.4}
\|\log q(x,a)\|< K\cdot \dist(x,ax)^\beta
\end{equation}
Here $\beta$ is the H\"older exponent for the bundle $E$ and $\cE_{i-1}, \cE_{i}$. We need the following lemma.
\begin{lemma}[Lemma 3.6 of \cite{KS07}]\label{lemma: L3.6 KS07}Let $W$ be a coarse Lyapunov foliation and $E=TW$. Let $a_0$ be a generic singular element in the corresponding Lyapunov hyperplane. Then there exist positive constants $\epsilon_0, C, \lambda$ such that for any $x\in \tM$ and $a=ta_0$ with $\dist(x, \tal(a)\cdot x)=\epsilon<\epsilon_0$, there exist a $y\in \tM$ and a $\delta\in \RR^k$ such that 
\begin{enumerate}
\item $\dist(x,y)<C\epsilon$;
\item $\tal(a+\delta)\cdot y\in W(y)$;
\item $\|\delta\|< C\epsilon$l
\item $\dist(\tal(sa_0)\cdot x, \tal(sa_0)\cdot y)<C\epsilon e^{-\lambda\min\{s,t-s\}}$ for any $0\leq s\leq t$;
\item $\|\log q(x,a)-q(y,a)\|<C\epsilon^\beta$, where $q$ is defined in \eqref{eqn: def q clos}.
\end{enumerate}
\end{lemma}
\begin{proof}The proof of inequalities (1), (2), (3), (4) of Lemma \ref{lemma: L3.6 KS07} are exactly the same as that of Lemma 3.6 in \cite{KS07}. Although the definition of $q$ here is slightly different from that of in \cite{KS07}, but as in \cite{KS07}, (5). could be proved from the standard estimate for a H\"older contraction coefficient along exponentially close orbits. 
\end{proof}
We come back to the proof of Lemma \ref{lemma: coh eqn closing}. Our notation is as in Lemma \ref{lemma: L3.6 KS07}. Let $b=a+\delta$, since $\delta $ is small we have $|\log q(y,a)-\log q(y,b)|<C_1\epsilon$. Then by (5). of Lemma \ref{lemma: L3.6 KS07} we have $|\log q(x,a)-\log q(y,b)|<C_2\epsilon^\beta$. Therefore to prove \eqref{eqn: KS07 3.4} we only need to show that 
\begin{equation}\label{eqn: KS07 3.6}
|\log q(y,b)|<C_3\epsilon^\beta 
\end{equation}
As in \cite{KS07}, we take an Anosov element $c$ contracts $W$, and pick $y_\ast:=\lim \tal(t_nc)\cdot y$ to be an accumulation point of the $c-$orbit of $y$. Then we have $b(y_\ast)=y$. In fact since $\tal(b)\cdot y\in W(y)\subset W^s_c(y)$ we have $\tal(b)\cdot y_\ast=\lim \tal(b)(\tal(tc_n)
\cdot y)=\lim \tal(t_nc)(\tal(b)\cdot y)=y_\ast$.

Our goal is to prove the following estimate which imply \eqref{eqn: KS07 3.6}.
\begin{enumerate}
\item $q(y_\ast, b)$ is close to $1$.
\item $q(y_\ast, b)$ is close to $q(y,b)$.
\end{enumerate}

Since $\delta$ is small, we could assume that $b$ is not contained in any Lyapunov hyperplane except possibly for the Lyapunov hyperplane $\mathcal{L}$ corresponding to $W$. 
\begin{itemize}
\item If $b\in \mathcal{L}$, we claim that $q(y_\ast,b)=1$. If it is not the case, since $y_\ast$ is $\tal(b)$ fixed point and $\cE_i(y_\ast), \cE_{i-1}(y_\ast)$ are $D\tal(b)-$invariant, we have that $q(y_\ast,tb)=q(y_\ast,b)^t$, hence $q(y_\ast,tb)$ will tend to $\infty$ or $0$ exponentially fast, which contradicts with Lemma \ref{lemma:sub exp grwth on plane}\footnote{Exponential growth (or decrease) of $q(y_\ast,tb)$ implies $D\tal(b)|_E(y_\ast)$ has an eigenvalue with norm not equal to $1$, which contradicts with Lemma \ref{lemma:sub exp grwth on plane}}.
\item If $b\notin \mathcal{L}$, then by Proposition $\ref{prop: properties suspen}$ we know $b$ is regular and hence Anosov. As \cite{KS07}, by \cite{Q} we know the orbit $\tal(\RR^k)\cdot y_\ast$ of its fixed point $y_\ast$ is compact; the Lyapunov exponents and the Lyapunov splitting are defined everywhere on this compact orbit. By Lemma \ref{lemma: hold cont cEi} we know $\tal$ preserves a H\"older continuous conformal structure within $\cE_i/\cE_{i-1}$, hence there is only one Lyapunov exponent on this orbit. We denote this exponent by $\tilde{\chi}$. Since $a\in \mathcal{L}$, by Lemma \ref{lemma:sub exp grwth on plane} we know $\tilde{\chi}(a)=0$, then $$|\tilde\chi(b)|=|\tilde\chi(a)+\tilde\chi(b)|=|\tilde\chi(\delta)|<C_4\|\delta\|<C_5\epsilon$$
But since $y_\ast$ is a fixed point for $\tal(b)$, then $\tilde\chi(b)$ is actually the logarithm of eigenvalues of $\tal(b)$ at $y_\ast$, hence we get 
\begin{equation}\label{eqn: 3.7 KS07}
|\log q(y_\ast,b)|<C_5'\epsilon
\end{equation}
\end{itemize} 
Therefore in both cases $b\in \mathcal{L}$ or $b\notin \mathcal{L}$, the equation \eqref{eqn: 3.7 KS07} holds.

Now we will prove that 
\begin{equation}\label{eqn: 3.8 KS07}
|\log(q(y,b))-\log q(y_\ast,b)|\leq C_6\epsilon^\beta
\end{equation}
In fact by $\tal(b)=\tal(-t_nc)\circ \tal(b)\circ\tal(t_nc)$ we have 
$$q(y,b)=q(\tal(b+t_nc)\cdot y, -t_nc)\cdot q(\tal(t_nc)\cdot y, b)\cdot q(y, t_nc)$$
The middle term tends to $q(y_\ast,b)$ since $\tal(t_nc)\cdot y\to y_\ast$. Moreover we claim that the product of the other terms is close to $1$:
$$\log [q(\tal(b+t_nc)\cdot y, -t_nc)\cdot q(y, t_nc)]=|\log q(\tal(b)\cdot y, t_nc)-\log q(y, t_nc)|<C_7 \epsilon^\beta $$
where the last inequailty follows from the standard estimate for a H\"older contraction coefficient along exponentially close orbits since $\tal(b)\cdot y\in W(y)\subset W^s_c(y)$. Therefore we get \eqref{eqn: 3.8 KS07}, \eqref{eqn: 3.8 KS07} and \eqref{eqn: 3.7 KS07} implies \eqref{eqn: KS07 3.6}. Therefore we get \eqref{eqn: KS07 3.4}, which completes the proof
\end{proof}



\end{document}

%% file: makros.tex
\newtheorem{maintheorem}{Theorem} 
\newtheorem{theorem}{Theorem}
\newtheorem{lemma}[theorem]{Lemma}
\newtheorem{coro}[theorem]{Corollary}
\newtheorem{prop}[theorem]{Proposition}

\newtheorem*{assumptions*}{Assumptions}

\newtheorem*{rem*}{Remark}

\theoremstyle{remark}
\newtheorem{remark}[theorem]{Remark}
\newtheorem*{remark*}{Remark}

\theoremstyle{definition}
\newtheorem{definition}{Definition}

\newcommand{\A}{{\mathbf A}}

\newcommand{\C}{{\mathbf C}}

\newcommand{\W}{{\mathbf W}}

\newcommand{\PP}{{\mathbb P}}

\newcommand{\RR}{{\mathbb R}}

\newcommand{\ZZ}{{\mathbb Z}}

\newcommand{\be}[1]{\begin{equation} \label{#1} }
\newcommand{\ee}{\end{equation}}
\newcommand{\beq}{\begin{equation}}